\documentclass[12pt]{amsart}
\usepackage{amssymb}
\usepackage[all]{xy}

\usepackage{url}
\usepackage{geometry}
\usepackage{enumerate}
\usepackage{graphicx}
\usepackage{amssymb,amsthm}
\usepackage{mathtools}
\usepackage{optidef}
\RequirePackage[colorlinks=true,citecolor=blue,linkcolor=blue,urlcolor=blue]{hyperref}
\usepackage{tikz}
\DeclarePairedDelimiter\ceil{\lceil}{\rceil}
\DeclarePairedDelimiter\floor{\lfloor}{\rfloor}
\DeclareMathOperator{\Z}{Z}
\newcommand\abs[1]{\left|#1\right|}
\newcommand{\iu}{{i\mkern1mu}}
\DeclarePairedDelimiter{\nint}\lfloor\rceil

\newtheorem{theorem}{Theorem}
\newtheorem{lemma}[theorem]{Lemma}
\newtheorem{corollary}[theorem]{Corollary}
\newtheorem{proposition}[theorem]{Proposition}
\newtheorem{conjecture}[theorem]{Conjecture}
\theoremstyle{remark}

\theoremstyle{definition}

\author[Becker]{Paul Becker}
\address{Department of Mathematics\\
         Penn State Behrend\\
         Erie, PA 16563
}
\email{\href{mailto:peb8@psu.edu}{peb8@psu.edu}}

\author[Cameron]{Thomas R.~Cameron}
\address{Department of Mathematics\\
         Penn State Behrend\\
         Erie, PA 16563
}
\email{\href{mailto:trc5475@psu.edu}{trc5475@psu.edu}}

\author[Hanely]{Derek Hanely}
\address{Department of Mathematics\\
         Penn State Behrend\\
         Erie, PA 16563
}
\email{\href{mailto:dwh28@psu.edu}{dwh28@psu.edu}}

\author[Ong]{Boon Ong}
\address{Department of Mathematics\\
         Penn State Behrend\\
         Erie, PA 16563
}
\email{\href{mailto:bwo1@psu.edu}{bwo1@psu.edu}}

\author[Previte]{Joseph P.~Previte}
\address{Department of Mathematics\\
         Penn State Behrend\\
         Erie, PA 16563
}
\email{\href{mailto:jpp4@psu.edu}{jpp4@psu.edu}}

\date{8 April 2024}

\title[On the number of minimal forts of a graph]{On the number of minimal forts of a graph} 
\keywords{forts, minimal forts, zero forcing, graph products, integer programming}
\subjclass[MSC Classification]{05C15, 05C30, 05C35, 05C50, 05C57, 05C76, 90C35}

\begin{document}

\begin{abstract}
In 2018, a fort of a graph was introduced as a non-empty subset of vertices in which no vertex outside of the set has exactly one neighbor in the set. Since then, forts have been used to characterize zero forcing sets, model the zero forcing number as an integer program, and generate lower bounds on the zero forcing number of a Cartesian product. In this article, we investigate the number of minimal forts of a graph, where a fort is minimal if every proper subset is not a fort. In particular, we show that the number of minimal forts of a graph of order at least six is strictly less than Sperner's bound, a famous bound due to Emanuel Sperner (1928) on the size of a collection of subsets where no subset contains another. Then, we derive an explicit formula for the number of minimal forts for several families of graphs, including the path, cycle, and spider graphs. Moreover, we show that the asymptotic growth rate of the number of minimal forts of the spider graph is bounded above by that of the path graph. We conjecture that the asymptotic growth rate of the path graph is extremal over all trees. Finally, we develop methods for constructing minimal forts of graph products using the minimal forts of the original graphs. In the process, we derive explicit formulas and lower bounds on the number of minimal forts for additional families of graphs, such as the wheel, sunlet, and windmill graphs. Most notably, we show that the family of windmill graphs has an exponential number of minimal forts with a maximum asymptotic growth rate of cube root of three, which is the largest asymptotic growth rate we have observed. We conjecture that there exist families of graphs with a larger asymptotic growth rate.
\end{abstract}

\maketitle

\section{Introduction}\label{sec:intro}
Zero forcing is a binary coloring game on a graph where a set of filled vertices can force non-filled vertices to become filled following a color change rule. 
In 2007, zero forcing was independently developed by Burgarth and Giovannetti in the control of quantum systems~\cite{Burgarth2007}, where it was called graph infection.
In 2008, the zero forcing number was shown to be an upper bound on the maximum nullity of a real symmetric matrix associated with a graph~\cite{AIM2008}. 

The computation of the zero forcing number has been shown to be a NP-complete problem~\cite{Aazami2008}.
Since then, there have been several integer programming models developed for computing the zero forcing number~\cite{Agra2019,Brimkov2019,Brimkov2021}. 
One such model is referenced as the fort cover model since the constraints of the model are built upon the forts of a graph.
The forts of a graph was originally introduced in~\cite{Fast2018} and is closely related to the failed zero forcing sets of a graph~\cite{Abara2022,Fetcie2015,Swanson2023}. 
In particular, the complement of a stalled failed zero forcing set is a fort, see the \emph{closure complement method} in~\cite{Brimkov2019}.

It is natural to consider the number of minimal forts a graph can have, where a fort is minimal if every proper subset is not a fort. 
For instance, the constraints in the fort cover model could be reduced to the minimal forts of a graph. 
Moreover, the minimal forts play an important role in characterizing the facet-defining forts of the zero forcing polytope, which is the convex hull of the feasible solutions to the fort cover model.
Finally, the bound in~\cite[Corollary 4.24]{Cameron2023} depends on the maximum number of minimal forts containing a single vertex of the graph.

In this article, we show that the number of minimal forts of a graph of order $n\geq 6$ is strictly less than $\binom{n}{\floor{n/2}}$.
Thus, the number of minimal forts does not attain \emph{Sperner's bound}, a famous bound on the size of a collection of subsets where no subset contains another~\cite{Sperner1928}. 
We also derive explicit formulas for the number of minimal forts for several families of graphs, including the path, cycle, and spider graphs.
In particular, for all spider graphs, we show that the number of minimal forts has an asymptotic growth rate bounded above by the asymptotic growth rate of the path graph. 
We suspect that the asymptotic growth rate of the path graph is extremal over all trees. 

Then, we develop methods for constructing minimal forts of graph products using the minimal forts of the original graphs.
Moreover, we use these construction methods to provide lower bounds on the number of minimal forts for several families of graphs, such as the wheel, sunlet, and windmill graphs. 
In particular, we show that the windmill graph $Wd(4,k)$ of order $n=3k+1$ has at least $3^{(n-1)/3}$ minimal forts, which is the largest asymptotic growth rate we have observed.
Though, we suspect that there exist families of graphs with a larger asymptotic growth rate. 

Before proceeding, we note that the authors in~\cite{Brimkov2019,Brimkov2021} considered the similar problem of constructing facet-defining forts for the zero forcing polytope. 
When appropriate, we make note of any overlap between their results and ours. 
In addition, several authors have considered the similar problem of constructing maximal failed zero forcing sets for graphs and graph products~\cite{Abara2022,Fetcie2015,Swanson2023}. 
However, these constructions are generally focused on a single maximal failed zero forcing set in order to obtain lower bounds on the failed zero forcing number.
\section{Preliminaries}\label{sec:prelim}
Throughout this article, we let $\mathbb{G}$ denote the set of all finite simple unweighted graphs.
For each $G\in\mathbb{G}$, we have $G=(V,E)$, where $V$ is the vertex set, $E$ is the edge set, and $\{u,v\}\in E$ if and only if $u\neq v$ and there is an edge between vertices $u$ and $v$.
If the context is not clear, we use $V(G)$ and $E(G)$ to specify the vertex set and edge set of $G$, respectively. 
The \emph{order} of $G$ is denoted by $n=\abs{V}$ and the \emph{size} of $G$ is denoted by $m=\abs{E}$; when $m=0$, we refer to $G$ as the \emph{empty graph}.

Let $G\in\mathbb{G}$.
Given $u\in V$, we define the \emph{neighborhood} of $u$ as 
\[
N(u) = \left\{v\in V\colon\{u,v\}\in E\right\}.
\]
We refer to every $v\in N(u)$ as a \emph{neighbor} of $u$, and we say that $u$ and $v$ are \emph{adjacent}. 
If the context is not clear, we use $N_{G}(u)$ to denote the neighborhood of $u$ in the graph $G$.
The \emph{closed neighborhood} of $u$ is defined by $N[u] = N(u)\cup\{u\}$. 
If $u,v\in V$ satisfy $N(u)=N(v)$, then we say that $u$ and $v$ are \emph{twins}. 
The degree of $u$ is denoted by $d(u)=\abs{N(u)}$; when $d(u)=0$, we refer to $u$ as an \emph{isolated vertex}, when $d(u)=1$, we refer to $u$ as a \emph{pendant vertex}, and when $d(u)\geq 3$, we refer to $u$ as a \emph{junction vertex}.
The \emph{minimum degree} of $G$ is denoted by $\delta(G)=\min\{d(u)\colon u\in V\}$.

The graph $G'=(V',E')$ is a \emph{subgraph} of $G$ if $V'\subseteq V$ and $E'\subseteq E$. 
Given $W\subseteq V$, the subgraph induced by $W$ is the subgraph of $G$ made up of the vertices in $W$ and the edges $\{u,v\}\in E$ such that $u,v\in W$. 
We reference this subgraph as an \emph{induced subgraph} and denote it by $G[W]$. 
If $G[W]$ is empty, then $W$ forms an \emph{independent set}. 
Furthermore, $W$ is a \emph{maximal independent set} if $W$ is not contained in another independent set. 

Given vertices $u,v\in V$, we say that $u$ is \emph{connected to} $v$ if there exists a list of vertices $(w_{0},w_{1},\ldots,w_{l})$ such that $u=w_{0}$, $v=w_{l}$, $\{w_{i},w_{i+1}\}\in E$ for all $i=1,\ldots,l-1$, and $w_{i}\neq w_{j}$ for all $i\neq j$. 
Such a list of vertices is called a \emph{path}; in particular, we reference it as a $(u,v)$-path. 
The connected to relation is an equivalence relation, which partitions the vertex set $V$ into disjoint subsets: $W_{1},W_{2},\ldots,W_{k}$, where for all $i=1,\ldots,k$ and for all $u,v\in W_{i}$, $u$ is connected to $v$. 

For each $i=1,\ldots,k$, the induced subgraph $G[W_{i}]$ is called a \emph{connected component} of $G$.
The graph $G$ is \emph{connected} if it has exactly one connected component. 
Given $u\in V$, the graph $G-u$ is the induced subgraph $G[V\setminus\{u\}]$.
Also, given $\{u,v\}\in E$, the graph $G-\{u,v\}$ is the subgraph of $G$ obtained by deleting the edge $\{u,v\}$. 
If $G-u$ has more connected components than $G$, then $u$ is a \emph{cut vertex} of $G$.
Similarly, if $G-\{u,v\}$ has more connected components than $G$, then $\{u,v\}$ is a \emph{cut edge} of $G$.

\emph{Zero forcing} is a binary coloring game on a graph, where vertices are either filled or non-filled. 
In this article, we denote filled vertices by the color gray and non-filled vertices by the color white. 
An initial set of gray vertices can force white vertices to become gray following a color change rule. 
While there are many color change rules, see~\cite[Chapter 9]{Hogben2022}, we will use the \emph{standard rule} which states that a gray vertex $u$ can force a white vertex $v$ if $v$ is the only white neighbor of $u$.
Since the vertex set is finite, there comes a point in which no more forcings are possible.
If at this point all vertices are gray, then we say that the initial set of gray vertices is a \emph{zero forcing set} of $G$; otherwise, we refer to the initial set of gray vertices as a \emph{failed zero forcing set} of $G$.
The \emph{zero forcing number} of $G$, denoted $\Z(G)$, is the minimum cardinality of a zero forcing set of $G$.

Given a graph $G\in\mathbb{G}$, a non-empty subset $F\subseteq V$ is a \emph{fort} if no vertex $u\in V\setminus{F}$ has exactly one neighbor in $F$.
Let $\mathcal{F}_{G}$ denote the collection of all forts of $G$. 
Note that $\mathcal{F}_{G}$ forms a cover for the zero forcing sets of $G$; in particular, Theorem~\ref{thm:fort-cover} states that a subset of vertices is a zero forcing set if and only if that subset intersects every fort. 
While one direction of this result was originally proven in~\cite[Theorem 3]{Fast2018}, both directions are shown in~\cite[Theorem 8]{Brimkov2019}.
\begin{theorem}[Theorem 8 in~\cite{Brimkov2019}]\label{thm:fort-cover}
Let $G\in\mathbb{G}$.
Then, $S\subseteq V$ is a zero forcing set of $G$ if and only if $S$ intersects every fort in $\mathcal{F}_{G}$.
\end{theorem}

The result in Theorem~\ref{thm:fort-cover} motivates the fort cover model for the zero forcing number of a graph, first proposed in~\cite{Brimkov2019}, which we state in~\eqref{eq:fc-obj}--\eqref{eq:fc-const2}.
\begin{mini!}
	{}{\sum_{v\in V}s_{v}}{}{}\label{eq:fc-obj}
	\addConstraint{\sum_{v\in F}s_{v}}{\geq 1,~\quad\forall F\in \mathcal{F}_{G}}\label{eq:fc-const1}
    \addConstraint{s_{v}}{\in\{0,1\},~\quad\forall v\in V}\label{eq:fc-const2}
 \end{mini!}

There exist families of graphs with an exponential number of forts.
For example, the complete graph $K_{n}$ has $2^{n}-(n+1)$ forts since every subset of cardinality $2$ or more is a fort of $K_{n}$.
Therefore, according to~\cite{Brimkov2019}, the fort cover model must use constraint generation, see~\cite{Dantzig1954,Nemhauser1999}.
In particular, a relaxed model is obtained from the full model by omitting the constraints in~\eqref{eq:fc-const1}; then, the relaxed model is solved and a set of violated constraints from the full model are added to the relaxed model, and this process is repeated until there are no more violated constraints.

However, $K_{n}$ only has $\binom{n}{2}$ minimal forts, where $F\in\mathcal{F}_{G}$ is a \emph{minimal fort} if every proper subset of $F$ is not a fort of $G$.
Since every fort contains a minimal fort, Theorem~\ref{thm:fort-cover} still holds if $\mathcal{F}_{G}$ is replaced with the collection of all minimal forts of $G$. 
Hence, the complete graph is not sufficient to justify the need for constraint generation in the fort cover model. 

In~\cite{Brimkov2019}, the authors define the \emph{zero forcing polytope} as the convex hull of all binary vectors that satisfy constraint~\eqref{eq:fc-const1}.
The \emph{facet-defining forts} are inequalities of the form in~\eqref{eq:fc-const1} that induce a facet of the zero forcing polytope.
These facet-defining forts correspond to minimal forts that satisfy additional constraints~\cite[Theorem 10]{Brimkov2019}.
Under suitable conditions, the number of facet-defining forts contained in the union of exactly two paths incident to a junction vertex, or exactly one cycle incident on a junction vertex, follows the Padovan sequence~\cite[Theorem 1 and Example 2]{Brimkov2021}. 
Since the Padovan sequence grows exponentially, this result justifies the need for constraint generation in the fort cover model. 
To our knowledge, this is the first family of graphs identified in the literature with an exponential number of minimal forts. 
We have two similar results, see Corollaries~\ref{cor:cycle_graph} and~\ref{cor:path_graph}, where the proof and statement of our results has the advantage of simplicity. 

Minimal forts also play a role in developing lower bounds on the zero forcing number of a Cartesian product. 
For instance, consider Theorem~\ref{thm:zf_cart_bound} which is proved in~\cite{Cameron2023}.
In this article, we identify several families of graphs for which the bounds in this theorem are not very good since these graphs have an exponential number of forts that contain a single vertex.
\begin{theorem}[Corollary 4.24 in~\cite{Cameron2023}]\label{thm:zf_cart_bound}
Let $G'\in\mathbb{G}$ be non-empty and suppose that every vertex of $G'$ is in at most $\Delta$ minimal forts of $G'$.
Then, for all non-empty $G\in\mathbb{G}$, 
\[
\Z\left(G\Box G'\right) \geq \frac{\Z(G)\Z(G')}{1+\log(\Delta)} + 1.
\]
\end{theorem}
\section{Sperner's Bound}\label{sec:sperner_bound}
Let $G\in\mathbb{G}$ and let $\mathcal{F}_{G}$ denote the collection of minimal forts of $G$.
Note that $\mathcal{F}_{G}$ is an example of a clutter of $V$.
In general, a \emph{clutter} of a finite set $S$ is a collection of subsets of $S$ such that no subset in the collection is contained in another~\cite{Brualdi1992}.
The following famous theorem due to Emanuel Sperner provides a sharp upper bound on the number of subsets in a clutter; see~\cite{Sperner1928} for the original article in German. 

\begin{theorem}[Theorem 5.4.3 of~\cite{Brualdi1992}]\label{thm:sperner}
Let $S$ be a finite set with cardinality $n$.
Then a clutter on $S$ contains at most $\binom{n}{\floor{n/2}}$ sets.
Furthermore, when $n$ is even, the only clutter of size $\binom{n}{n/2}$ is the clutter of all subsets of $S$ of cardinality $n/2$; if $n$ is odd, the only clutters of size $\binom{n}{\floor{n/2}}$ is the clutter of all subsets of $S$ of cardinality $\floor{n/2}$ and the clutter of all subsets of $S$ of cardinality $\ceil{n/2}$.   
\end{theorem}

It is immediately clear from Theorem~\ref{thm:sperner} that $\abs{\mathcal{F}_{G}}\leq\binom{n}{\floor{n/2}}$.
In this section, we show that this bound is not sharp for collections of minimal forts for graphs of order $n\geq 6$. 
We begin with the following lemma. 
\begin{lemma}\label{lem:high_deg}
Let $G\in\mathbb{G}$ be a graph with order $n\geq 4$ and no isolated vertices. 
Then, every subset of vertices with cardinality $\floor{n/2}$ is a fort of $G$ if and only if the minimum degree satisfies $\delta(G)\geq\ceil{n/2}+1$.
Similarly, every subset of vertices with cardinality $\ceil{n/2}$ is a fort of $G$ if and only if the minimum degree satisfies $\delta(G)\geq\floor{n/2}+1$.
\end{lemma}
\begin{proof}
We prove the first if and only if statement and note that the second statement can be shown using a similar argument. 

Suppose there exists a vertex $u\in V$ such that
\[
1 \leq \abs{N(u)} \leq \ceil{n/2}.
\]
Then, construct $F\subset (V\setminus \{u\})$ as follows: $F$ contains exactly $1$ vertex from $N(u)$ and exactly $\left(\floor{n/2}-1\right)$ vertices not from $N[u]$.
Since $u$ has exactly one neighbor in $F$ and $\abs{F}=\floor{n/2}$, it follows that not every subset of vertices with cardinality $\floor{n/2}$ is a fort of $G$. 

Conversely, suppose that every vertex of $G$ has degree at least $\ceil{n/2}+1$.
Choose any set $F \subseteq V$ with $\lfloor\frac{n}{2}\rfloor$ vertices, and any $u \notin F.$ 
The set $F$ intersects $N(u)$ at least twice; hence, $F$ is a fort of $G$.
Therefore, every subset of vertices with cardinality $\floor{n/2}$ is a fort of $G$.
\end{proof}

Note that the order $n\geq 4$ is necessary in Lemma~\ref{lem:high_deg} for the minimum degree to satisfy $\delta(G)\geq\ceil{n/2}+1$.
Moreover, the graph $K_{1}\cup K_{5}$ illustrates the need to avoid isolated vertices since every subset of cardinality $n/2=3$ is a fort but the minimum degree is $0$. 
Next, we show that a minimum degree of $\delta(G)\geq\ceil{n/2}+1$ or $\delta(G)\geq\floor{n/2}+1$ is sufficient to guarantee the existence of a fort with cardinality less than $\floor{n/2}$ or $\ceil{n/2}$, respectively.
\begin{theorem}\label{thm:smaller_fort}
Let $G\in\mathbb{G}$ have order $n\geq 6$.
If $\delta(G)\geq\ceil{n/2}+1$, then there is a fort of $G$ with cardinality less than $\floor{n/2}$.
Similarly, if $\delta(G)\geq\floor{n/2}+1$, then there is a fort of $G$ with cardinality less than $\ceil{n/2}$.
\end{theorem}
\begin{proof}
We prove the first implication and note that the second statement can be shown using a similar argument. 

Let $S\subset V$ have cardinality $\floor{n/2}-1$. 
Note that, since $n\geq 6$, $\abs{S}\geq 2$.
If $S$ is a fort of $G$, then we are done. 
Otherwise, there is a $u\notin S$ such that $\abs{N(u)\cap S}=1$.
Define $v$ to be the unique vertex in $N(u)\cap S$.
Also, define $\hat{S}=V\setminus\{S\cup\{u\}\}$ and note that $\abs{\hat{S}}=\ceil{n/2}$.
Since $d(u)\geq\ceil{n/2}+1$ and $\abs{N(u)\cap S}=1$, it follows that $u$ is adjacent to every vertex in $\hat{S}$ and every vertex in $\hat{S}$ is adjacent to at least one vertex in $S$.
Now, we proceed via cases to construct a fort of cardinality less than $\floor{n/2}$.

\underline{Case 1}: Suppose there exists a $w\in\hat{S}$ such that $N(w)\cap S=\{v\}$.
Then, $w$ is adjacent to every vertex in $\hat{S}$, which implies that $u$ and $w$ are twins. 
Hence, $\{u,w\}$ is a fort of $G$.
Since $n\geq 6$, we have identified a fort of cardinality less than $\floor{n/2}$. 

\underline{Case 2}: Suppose that $N(v)\cap S=\emptyset$. Then, $N(v)=\hat{S}\cup\{u\}$, which implies that $u$ and $v$ are twins. 
Hence, $\{u,v\}$ is a fort of $G$.
Since $n\geq 6$, we have identified a fort of cardinality less than $\floor{n/2}$.

\underline{Case 3}: If neither Case 1 nor Case 2 hold, then $F=\left(S\setminus\{v\}\right)\cup\{u\}$ is a fort of $G$ with cardinality $\floor{n/2}-1$.
Indeed, since Case 2 does not hold, it follows that $v$ is adjacent to at least one vertex in $S$; hence, $\abs{N(v)\cap F}\geq 2$. 
Furthermore, since Case 1 does not hold, every $w\in\hat{S}$ that is adjacent to $v$ is adjacent to at least one other vertex in $S$; hence, $\abs{N(w)\cap F}\geq 2$ for all $w\in\hat{S}$. 
\end{proof}

Note that the complete graph of order $n=5$ has minimum degree $\ceil{n/2}+1=4$ but no fort of size less than $\floor{n/2}=2$. 
Hence, the order $n\geq 6$ is the best possible in Theorem~\ref{thm:smaller_fort}.
Finally, we show that Sperner's bound is not sharp for collections of minimal forts of a graph of order $n\geq 6$.
\begin{corollary}\label{cor:sperner_fort}
Let $G\in\mathbb{G}$ have order $n\geq 6$, and let $\mathcal{F}_{G}$ denote the collection of minimal forts of $G$.
Then, $\abs{\mathcal{F}_{G}}<\binom{n}{\floor{n/2}}$.
\end{corollary}
\begin{proof}
For the sake of contradiction, suppose that $\abs{\mathcal{F}_{G}}=\binom{n}{\floor{n/2}}$.
If $n$ is even, then Theorem~\ref{thm:sperner} implies that every subset of $V$ of cardinality $n/2$ is a minimal fort of $G$.
Since isolated vertices constitute a fort of cardinality $1$, it follows that $G$ must have no isolated vertices.  
Hence, by Lemma~\ref{lem:high_deg}, the minimum degree satisfies $\delta(G)\geq n/2+1$.
Furthermore, by Theorem~\ref{thm:smaller_fort}, there is a fort of $G$ with cardinality less than $n/2$, which contradicts every subset of $V$ of cardinality $n/2$ being a minimal fort of $G$.

If $n$ is odd, then Theorem~\ref{thm:sperner} implies that every subset of $V$ of cardinality $\floor{n/2}$ is a minimal fort of $G$, or every subset of $V$ of cardinality $\ceil{n/2}$ is a minimal fort of $G$. 
Again, $G$ cannot have any isolated vertices. 
Hence, in the former case, Lemma~\ref{lem:high_deg} implies that the minimum degree satisfies $\delta(G)\geq\ceil{n/2}+1$. 
Furthermore, by Theorem~\ref{thm:smaller_fort}, there is a fort of $G$ with cardinality less than $\floor{n/2}$, which contradicts every subsets of $V$ of cardinality $\floor{n/2}$ being a minimal fort of $G$. 
In the latter case, Lemma~\ref{lem:high_deg} implies that the minimum degree satisfies $\delta(G)\geq\floor{n/2}+1$.
Furthermore, by Theorem~\ref{thm:smaller_fort}, there is a fort of $G$ with cardinality less than $\ceil{n/2}$, which contradicts every subsets of $V$ of cardinality $\ceil{n/2}$ being a minimal fort of $G$. 
\end{proof}

Note that the complete graph of order $n=5$ has $\binom{n}{\floor{n/2}}=10$ minimal forts.
Hence, the order $n\geq 6$ is the best possible in Corollary~\ref{cor:sperner_fort}. 
\section{Families of Graphs}\label{sec:graph_family}
In this section, we identify the number of minimal forts for several families of graphs. 
Throughout, we let $\mathcal{F}_{G}$ denote the collection of all minimal forts of a graph $G\in\mathbb{G}$. 
We begin with basic results regarding the empty, complete, and complete bipartite graphs.
Recall that an \emph{empty graph} of order $n$, denoted $E_{n}$, is a graph with no edges; a \emph{complete graph} of order $n$, denoted $K_{n}$, is a graph with all $\binom{n}{2}$ edges; and a \emph{complete bipartite graph} $K_{p,q}$, where $p,q\geq 1$, is a graph made up of two partitions of vertices $V_{1},V_{2}$ such that $\abs{V_{1}}=p$, $\abs{V_{2}}=q$, and $\{u,v\}\in E$ if and only if $u\in V_{1}$ and $v\in V_{2}$. 
\begin{proposition}\label{prop:empty_graph}
Let $n\geq 1$.
Then, $\abs{\mathcal{F}_{E_{n}}} = n$. 
\end{proposition}
\begin{proof}
Note that a single vertex is a fort if and only if that vertex is isolated. 
Hence, there are $n$ minimal forts for $E_{n}$.
\end{proof}
\begin{proposition}\label{prop:complete_graph}
Let $n\geq 2$.
Then, $\abs{\mathcal{F}_{K_{n}}} = \binom{n}{2}$. 
\end{proposition}
\begin{proof}
Note that any pair of vertices is a fort if and only if those vertices are twins.
Hence, every pair of vertices in $K_{n}$ is a fort. 
Since no vertex is isolated, it follows that there are $\binom{n}{2}$ minimal forts for $K_{n}$.
\end{proof}
\begin{proposition}\label{prop:complete_bipartite}
Let $p,q\geq 1$ and $p+q\geq 3$. 
Then, $\abs{\mathcal{F}_{K_{p,q}}} = \binom{p}{2} + \binom{q}{2}$.
\end{proposition}
\begin{proof}
Let the vertices of $K_{p,q}$ be partitioned into sets $V_{1}$ and $V_{2},$ such that $\abs{V_{1}}=p$, $\abs{V_{2}}=q$, and $\{u,v\}\in E$ if and only if $u\in V_{1}$ and $v\in V_{2}$. 
Since every pair of vertices in $V_{1}$ are twins and no vertex is isolated, it follows that any pair of vertices in $V_{1}$ is a minimal fort.
Similarly, every pair of vertices in $V_{2}$ is a minimal fort.
Therefore, there are $\binom{p}{2}+\binom{q}{2}$ minimal forts for $K_{p,q}$.
\end{proof}

Next, we show that the number of minimal forts of the cycle graph follows the Perrin sequence~\cite[A001608]{OEIS} and the number of minimal forts for the path graph follows the Padovan sequence~\cite[A000931]{OEIS}, see Corollary~\ref{cor:cycle_graph} and Corollary~\ref{cor:path_graph}, respectively. 
Recall that a \emph{path graph}, denoted $P_{n}$, is a graph with vertex set $V=\{v_{1},\ldots,v_{n}\}$ and edge set $E=\left\{\{v_{i},v_{i+1}\}\colon i=1,\ldots,n-1\right\}$; a \emph{cycle graph}, denoted $C_{n}$, is obtained from a path $P_{n}$ by adding the edge $\{v_{1},v_{n}\}$.

Before proceeding, we discuss the related result in~\cite[Theorem 1]{Brimkov2021} where the authors provide explicit formulas for the number of facet-defining forts contained in the union of exactly two paths incident on a junction vertex, or exactly one cycle incident on a junction vertex. 
By~\cite[Theorem 10]{Brimkov2019}, facet-defining forts correspond to minimal forts that satisfy additional constraints; hence, these formulas can be used to determine a lower bound on the number of minimal forts contained in the union of exactly two paths incident on a junction vertex, or exactly one cycle incident on a junction vertex. 
However, since Corollaries~\ref{cor:cycle_graph} and~\ref{cor:path_graph} deal with cycle and path graphs directly, their statements and proofs have the advantage of simplicity.

In the proof of~\cite[Theorem 1]{Brimkov2021}, the authors implicitly utilize the characterization of a minimal fort as the complement of a maximal failed zero forcing set.
Throughout the remainder of this article, we also make use of the connection between minimal forts and maximal failed zero forcing sets. 
Since we are unaware of a proof of this characterization in the literature, we provide one here.
\begin{theorem}\label{thm:failed_zf}
Let $G\in\mathbb{G}$.
Then, $F\subseteq V$ is a minimal fort of $G$ if and only if $S=V\setminus{F}$ is a maximal failed zero forcing set of $G$.
\end{theorem}
\begin{proof}
Suppose that $F\subseteq V$ is a minimal fort of $G$ and define $S=V\setminus{F}$. 
Let all vertices in $S$ be gray and all vertices in $F$ be white. 
Then, for each $v\in S$, $\abs{N(v)\cap F}\neq 1$; hence, no forcings can occur so $S$ is a failed zero forcing set.
Since $F$ is a minimal fort of $G$, $S$ intersects every fort of $G$ except for $F$.
Hence, for any $u\in F$, $S\cup\{u\}$ intersects every fort of $G$ and, by Theorem~\ref{thm:fort-cover}, $S\cup\{u\}$ is a zero forcing set of $G$.
Therefore, $S$ is a maximal failed zero forcing set of $G$.

Conversely, suppose that $S\subseteq V$ is a maximal failed zero forcing set of $G$ and define $F=V\setminus{S}$.
Let all vertices in $S$ be gray and all vertices in $F$ be white.
Since no forcings can occur, it follows that each $v\in S$ satisfies $\abs{N(v)\cap F}\neq 1$; hence, $F$ is a fort of $G$. 
Furthermore, since $S$ is a maximal failed zero forcing set of $G$, it follows that for any non-empty $A\subseteq F$, $S\cup A$ is a zero forcing set.
Since $S\cup A$ does not intersect $F\setminus A$, Theorem~\ref{thm:fort-cover} implies that $F\setminus A$ is not a fort of~$G$.
Therefore, $F$ is a minimal fort of $G$. 
\end{proof}
\begin{corollary}\label{cor:cycle_graph}
For $n\geq 6$, the number of minimal forts for $C_{n}$ satisfies
\[
\abs{\mathcal{F}_{C_{n}}} = \abs{\mathcal{F}_{C_{n-2}}} + \abs{\mathcal{F}_{C_{n-3}}},
\]
where $\abs{\mathcal{F}_{C_{3}}}=3$, $\abs{\mathcal{F}_{C_{4}}}=2$, and $\abs{\mathcal{F}_{C_{5}}}=5$.
\end{corollary}
\begin{proof}
Every zero forcing set of $C_{n}$ requires at least $2$ adjacent vertices.
Hence, the maximal failed zero forcing sets of $C_{n}$ coincide with the maximal independent sets of $C_{n}$. 
It is known that the number of maximal independent sets of $C_{n}$ is the $n$th term of the Perrin sequence~\cite{Bisdorff2008,Furedi1987}.
Since the Perrin sequence follows the given recurrence relation and initial seeds, the result follows from Theorem~\ref{thm:failed_zf}.
\end{proof}
\begin{corollary}\label{cor:path_graph}
For $n\geq 4$, the number of minimal forts for $P_{n}$ satisfies
\[
\abs{\mathcal{F}_{P_{n}}} = \abs{\mathcal{F}_{P_{n-2}}} + \abs{\mathcal{F}_{P_{n-3}}},
\]
where $\abs{\mathcal{F}_{P_{1}}}=1$, $\abs{\mathcal{F}_{P_{2}}}=1$, and $\abs{\mathcal{F}_{P_{3}}}=1$.
\end{corollary}
\begin{proof}
Every zero forcing set of $P_{n}$ requires a pendant vertex or at least $2$ adjacent vertices.
Hence, maximal failed zero forcing sets of $P_{n}$ are maximal independent sets that do not include pendant vertices.

For $n=1$ and $n=2$, the only failed zero forcing set is the empty set.
For $n=3$, the maximal failed zero forcing set is $\{v_{2}\}$. 
For $n\geq 4$, every maximal failed zero forcing set contains $v_{2}$ or $v_{3}$, but not both.
If $v_{2}$ is in the maximal failed zero forcing set, then the remainder of the vertices can be chosen from a maximal failed zero forcing set for the path graph induced by the vertices $\{v_{3},\ldots,v_{n}\}$.
Similarly, if $v_{3}$ is in the maximal failed zero forcing set, then the remainder of the vertices can be chosen from a maximal failed zero forcing set for the path graph induced by the vertices $\{v_{4},\ldots,v_{n}\}$.

Now, the result follows from Theorem~\ref{thm:failed_zf}.
\end{proof}

Corollary~\ref{cor:cycle_graph} implies that the number of minimal forts of $C_{n}$ is the $n$th term of the Perrin sequence:
\[
x_{n} = x_{n-2} + x_{n-3},~n\geq 3,~x_{0}=3,~x_{1}=0,~x_{2}=2.
\]
Similarly, Corollary~\ref{cor:path_graph} implies that the number of minimal forts of $P_{n}$ is the $(n-1)$st term in the Padovan sequence:
\[
y_{n} = y_{n-2} + y_{n-3},~n\geq 3,~y_{0}=1,~y_{1}=1,~y_{2}=1.
\]
Both sequences have been extensively studied in the literature, for example, see~\cite{Lenoard2012,Minton2011,Shannon2006,Tedford2019}.
Moreover, both sequences can be written explicitly in terms of powers of the roots of the cubic $\lambda^{3}-\lambda-1$, which has one real root,
\[
\psi=1.3247179572\ldots,
\]
known as the \emph{plastic ratio}, see~\cite{Finch2003}, and two complex roots $\omega=re^{\iu\theta}$ and $\overline{\omega}=re^{-\iu\theta}$, where
\[
r = 0.8688369618\ldots~\text{and}~\theta = 2.4377349322\ldots.
\]
Given these roots, both the Perrin and Padovan sequences, respectively, have the following explicit formula:
\begin{align*}
x_{n} &= \psi^{n} + \omega^{n} + \overline{\omega}^{n} \\
y_{n} &= \alpha\psi^{n} + \beta\omega^{n} + \overline{\beta}\overline{\omega}^{n},
\end{align*}
where 
\[
\alpha = \frac{\psi^{5}}{2\psi + 3}~\textrm{and}~\beta = \frac{\omega^{5}}{2\omega + 3}.
\]
Since $\abs{\omega}=\abs{\overline{\omega}}<1$, it follows that the asymptotic value of both $x_{n}$ and $y_{n}$ is determined by the plastic ratio $\psi$.
In fact, we have $x_{n} = \nint{\psi^{n}}$, for $n\geq 10$, and $y_{n}=\nint{\frac{\psi^{n+5}}{2\psi+3}}$, for $n\geq 0$, where $\nint{\cdot}$ denotes rounding to the nearest integer. 
In the following corollaries, we summarize these results as it pertains to the number of minimal forts of the cycle and path graphs.
\begin{corollary}\label{cor:cycle_graph_explicit}
The number of minimal forts for $C_{n}$ satisfies
\[
\abs{\mathcal{F}_{C_{n}}} = \nint{\psi^{n}},~n\geq 10.
\]
Moreover, the limit of successive ratios satisfies
\[
\lim_{n\rightarrow\infty}\frac{\abs{\mathcal{F}_{C_{n+1}}}}{\abs{\mathcal{F}_{C_{n}}}}
=
\psi,
\]
where $\psi=1.3247179572\ldots$ denotes the plastic ratio.
\end{corollary}
\begin{corollary}\label{cor:path_graph_explicit}
The number of minimal forts for $P_{n}$ satisfies
\[
\abs{\mathcal{F}_{P_{n}}} = \nint{\frac{\psi^{n+4}}{2\psi+3}},~n\geq 1.
\]
Moreover, the limit of successive ratios satisfies
\[
\lim_{n\rightarrow\infty}\frac{\abs{\mathcal{F}_{P_{n+1}}}}{\abs{\mathcal{F}_{P_{n}}}}
=
\psi,
\]
where $\psi=1.3247179572\ldots$ denotes the plastic ratio.
\end{corollary}

It is worth noting that every minimal fort of $P_{n}$ contains both pendant vertices; hence, the path graph is an example of a family of graphs for which the bounds in Theorem~\ref{thm:zf_cart_bound} are not very good. 

We conclude this section by considering the family of spider graphs, which consists of trees with exactly one junction vertex. 
Recall that a \emph{tree} is a connected acyclic graph.
Moreover, a tree has the following equivalent characterizations. 
\begin{theorem}[Theorem 1.5.1 of~\cite{Diestel2016}]\label{thm:tree_char}
The following assertions are equivalent for a graph $G\in\mathbb{G}$.
\begin{enumerate}[(i.)]
\item $G=(V,E)$ is a tree.
\item For any $u,v\in V$, there is a unique $(u,v)$-path.
\item Every edge $\{u,v\}\in E$ is a cut edge.
\end{enumerate}
\end{theorem}

The following theorem shows that any pair of vertices in a minimal fort of a tree induce a minimal fort on a path.
\begin{theorem}\label{thm:tree_induced_path}
Let $T$ be a tree of order $n\geq 2$ and let $F\subseteq V(T)$ be a minimal fort of $T$.
Also, let $u,v\in F$ and let $W=\{w_{0},w_{1},\ldots,w_{\ell}\}$ denote the vertices in the unique $(u,v)$-path of $T$.
Then, $T[W]$ is an induced path graph of order $(\ell+1)$.
Moreover, $F'=F\cap W$ is a minimal fort of $T[W]$. 
\end{theorem}
\begin{proof}
Let $S=W\setminus{F}$.
We will show that $S$ is a maximal failed zero forcing set of $T[W]$.
Note that $u,v\notin S$; hence, all that remains is to show that no two adjacent vertices are in $S$.

For the sake of contradiction, suppose there is an $i\in\{1,\ldots,\ell-2\}$ such that $w_{i},w_{i+1}\in S$. 
Since $\{w_{i},w_{i+1}\}$ is a cut edge of $T$, it follows that $T-\{u_{i},u_{i+1}\}$ has two connected components induced by the disjoint vertex sets $W_{1}$ and $W_{2}$.
Without loss of generality, suppose that $u\in W_{1}$ and $v\in W_{2}$. 

Now, let $F_{1} = F\cap W_{1}$.
Since $u\in F\cap W_{1}$, it follows that $F_{1}$ is non-empty.
Furthermore, every vertex $w\in W_{1}\setminus F_{1}$ satisfies
\[
N_{T}(w)\cap F_{1} = N_{T}(w)\cap F,
\]
which implies that $F_{1}$ is a fort of $T$.

Similarly, one can show $F_{2} = F\cap W_{2}$ is a fort of $T$.
Since there is no vertex $w\in V(T)$ that is adjacent to both a vertex in $F_{1}$ and a vertex in $F_{2}$, it follows that $F$ is not a minimal fort. 
\end{proof}

In Figure~\ref{fig:forts_subpaths_Comb4}, we illustrate Theorem~\ref{thm:tree_induced_path} for a particular minimal fort of the $4$-\emph{comb graph}, a tree of order~$8$ obtained from the path graph~$P_4$ by joining a pendant vertex to each vertex in~$P_4$.
Note that each induced path (highlighted with dashed lines) itself induces a minimal fort on that path.

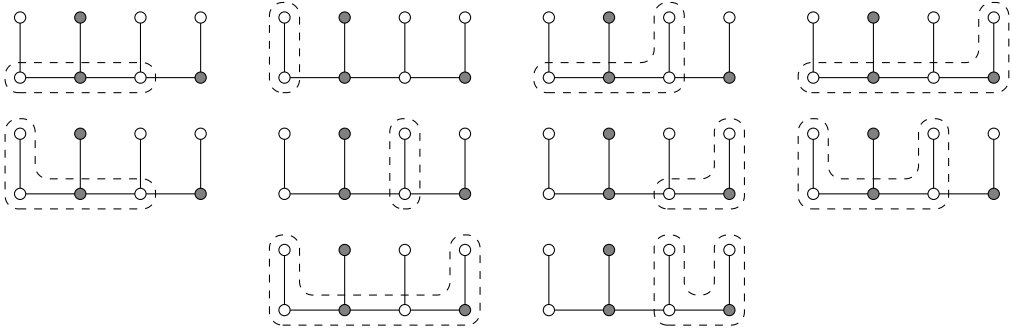
\begin{figure}[ht]
\centering
    \begin{tikzpicture}[every node/.style={circle,draw=black,inner sep=1.5}, scale=.8]
    \begin{scope}
    \draw[black] (0,1)--(0,0)--(3,0)--(3,1);
    \draw[black] (1,0)--(1,1);
    \draw[black] (2,0)--(2,1);
    \node[fill=white] at (0,0) {};
    \node[fill=gray] at (1,0) {};
    \node[fill=white] at (2,0) {};
    \node[fill=gray] at (3,0) {};
    \node[fill=white] at (0,1) {};
    \node[fill=gray] at (1,1) {};
    \node[fill=white] at (2,1) {};
    \node[fill=white] at (3,1) {};
    \draw[dashed] (-0.25,-0.25) [rounded corners=5pt] -- (2.25,-0.25)--(2.25,0.25)--(-.25,.25)--cycle;
    \end{scope}
    \begin{scope}[yshift=-55]
    \draw[black] (0,1)--(0,0)--(3,0)--(3,1);
    \draw[black] (1,0)--(1,1);
    \draw[black] (2,0)--(2,1);
    \node[fill=white] at (0,0) {};
    \node[fill=gray] at (1,0) {};
    \node[fill=white] at (2,0) {};
    \node[fill=gray] at (3,0) {};
    \node[fill=white] at (0,1) {};
    \node[fill=gray] at (1,1) {};
    \node[fill=white] at (2,1) {};
    \node[fill=white] at (3,1) {};
    \draw[dashed] (-0.25,-0.25) [rounded corners=5pt] -- (2.25,-0.25)--(2.25,.25)--(.25,.25)--(.25,1.25)--(-.25,1.25)--cycle;
    \end{scope}

    \begin{scope}[xshift=125]
    \draw[black] (0,1)--(0,0)--(3,0)--(3,1);
    \draw[black] (1,0)--(1,1);
    \draw[black] (2,0)--(2,1);
    \node[fill=white] at (0,0) {};
    \node[fill=gray] at (1,0) {};
    \node[fill=white] at (2,0) {};
    \node[fill=gray] at (3,0) {};
    \node[fill=white] at (0,1) {};
    \node[fill=gray] at (1,1) {};
    \node[fill=white] at (2,1) {};
    \node[fill=white] at (3,1) {};
    \draw[dashed] (-0.25,-0.25) [rounded corners=5pt] -- (.25,-0.25)--(.25,1.25)--(-.25,1.25)--cycle;
    \end{scope}
    \begin{scope}[xshift=125,yshift=-55]
    \draw[black] (0,1)--(0,0)--(3,0)--(3,1);
    \draw[black] (1,0)--(1,1);
    \draw[black] (2,0)--(2,1);
    \node[fill=white] at (0,0) {};
    \node[fill=gray] at (1,0) {};
    \node[fill=white] at (2,0) {};
    \node[fill=gray] at (3,0) {};
    \node[fill=white] at (0,1) {};
    \node[fill=gray] at (1,1) {};
    \node[fill=white] at (2,1) {};
    \node[fill=white] at (3,1) {};
    \draw[dashed] (1.75,-0.25) [rounded corners=5pt] -- (2.25,-0.25)--(2.25,1.25)--(1.75,1.25)--cycle;
    \end{scope}
    \begin{scope}[xshift=125,yshift=-110]
    \draw[black] (0,1)--(0,0)--(3,0)--(3,1);
    \draw[black] (1,0)--(1,1);
    \draw[black] (2,0)--(2,1);
    \node[fill=white] at (0,0) {};
    \node[fill=gray] at (1,0) {};
    \node[fill=white] at (2,0) {};
    \node[fill=gray] at (3,0) {};
    \node[fill=white] at (0,1) {};
    \node[fill=gray] at (1,1) {};
    \node[fill=white] at (2,1) {};
    \node[fill=white] at (3,1) {};
    \draw[dashed] (-0.25,-0.25) [rounded corners=5pt] -- (3.25,-0.25)--(3.25,1.25)--(2.75,1.25)--(2.75,.25)--(.25,.25)--(.25,1.25)--(-.25,1.25)--cycle;
    \end{scope}

    \begin{scope}[xshift=250]
    \draw[black] (0,1)--(0,0)--(3,0)--(3,1);
    \draw[black] (1,0)--(1,1);
    \draw[black] (2,0)--(2,1);
    \node[fill=white] at (0,0) {};
    \node[fill=gray] at (1,0) {};
    \node[fill=white] at (2,0) {};
    \node[fill=gray] at (3,0) {};
    \node[fill=white] at (0,1) {};
    \node[fill=gray] at (1,1) {};
    \node[fill=white] at (2,1) {};
    \node[fill=white] at (3,1) {};
    \draw[dashed] (-0.25,-0.25) [rounded corners=5pt] -- (2.25,-0.25)--(2.25,1.25)--(1.75,1.25)--(1.75,.25)--(-.25,.25)--cycle;
    \end{scope}
    \begin{scope}[xshift=250,yshift=-55]
    \draw[black] (0,1)--(0,0)--(3,0)--(3,1);
    \draw[black] (1,0)--(1,1);
    \draw[black] (2,0)--(2,1);
    \node[fill=white] at (0,0) {};
    \node[fill=gray] at (1,0) {};
    \node[fill=white] at (2,0) {};
    \node[fill=gray] at (3,0) {};
    \node[fill=white] at (0,1) {};
    \node[fill=gray] at (1,1) {};
    \node[fill=white] at (2,1) {};
    \node[fill=white] at (3,1) {};
    \draw[dashed] (1.75,-0.25) [rounded corners=5pt] -- (3.25,-0.25)--(3.25,1.25)--(2.75,1.25)--(2.75,.25)--(1.75,.25)--cycle;
    \end{scope}
    \begin{scope}[xshift=250,yshift=-110]
    \draw[black] (0,1)--(0,0)--(3,0)--(3,1);
    \draw[black] (1,0)--(1,1);
    \draw[black] (2,0)--(2,1);
    \node[fill=white] at (0,0) {};
    \node[fill=gray] at (1,0) {};
    \node[fill=white] at (2,0) {};
    \node[fill=gray] at (3,0) {};
    \node[fill=white] at (0,1) {};
    \node[fill=gray] at (1,1) {};
    \node[fill=white] at (2,1) {};
    \node[fill=white] at (3,1) {};
    \draw[dashed] (1.75,-0.25) [rounded corners=5pt] -- (3.25,-0.25)--(3.25,1.25)--(2.75,1.25)--(2.75,.25)--(2.25,.25)--(2.25,1.25)--(1.75,1.25)--cycle;
    \end{scope}

    \begin{scope}[xshift=375]
    \draw[black] (0,1)--(0,0)--(3,0)--(3,1);
    \draw[black] (1,0)--(1,1);
    \draw[black] (2,0)--(2,1);
    \node[fill=white] at (0,0) {};
    \node[fill=gray] at (1,0) {};
    \node[fill=white] at (2,0) {};
    \node[fill=gray] at (3,0) {};
    \node[fill=white] at (0,1) {};
    \node[fill=gray] at (1,1) {};
    \node[fill=white] at (2,1) {};
    \node[fill=white] at (3,1) {};
    \draw[dashed] (-0.25,-0.25) [rounded corners=5pt] -- (3.25,-0.25)--(3.25,1.25)--(2.75,1.25)--(2.75,.25)--(-.25,.25)--cycle;
    \end{scope}
    \begin{scope}[xshift=375,yshift=-55]
    \draw[black] (0,1)--(0,0)--(3,0)--(3,1);
    \draw[black] (1,0)--(1,1);
    \draw[black] (2,0)--(2,1);
    \node[fill=white] at (0,0) {};
    \node[fill=gray] at (1,0) {};
    \node[fill=white] at (2,0) {};
    \node[fill=gray] at (3,0) {};
    \node[fill=white] at (0,1) {};
    \node[fill=gray] at (1,1) {};
    \node[fill=white] at (2,1) {};
    \node[fill=white] at (3,1) {};
    \draw[dashed] (-0.25,-0.25) [rounded corners=5pt] -- (2.25,-0.25)--(2.25,1.25)--(1.75,1.25)--(1.75,.25)--(.25,.25)--(.25,1.25)--(-.25,1.25)--cycle;
    \end{scope}
    \end{tikzpicture}
    \caption{A minimal fort (white) and maximal failed zero forcing set (gray) of the $4$-comb graph, and its corresponding $\binom{5}{2}=10$ paths (indicated by dashed lines) induced by taking any pair of vertices in the minimal fort.}
    \label{fig:forts_subpaths_Comb4}
\end{figure}

Next, we derive an explicit formula for the number of minimal forts on a spider graph, which is a tree with exactly one junction vertex. 
In particular, we let $S_{l_{1},\ldots,l_{k}}$ denote a spider graph with junction vertex $v$, pendant vertices $u_{1},\ldots,u_{k}$, where $k\geq 2$ and for $i=1,\ldots,k$, the length of the path from $v$ to $u_{i}$ is $l_{i}\geq 1$. 
\begin{theorem}\label{thm:spider_graph}
Let $S_{l_{1},\ldots,l_{k}}$ denote a spider graph with junction vertex $v$, pendant vertices $u_{1},\ldots,u_{k}$, where $k\geq 2$ and for $i=1,\ldots,k$, the length of the path from $v$ to $u_{i}$ is $l_{i}\geq 1$. 
Then, the number of minimal forts satisfies
\[
\abs{\mathcal{F}_{S_{l_{1},\ldots,l_{k}}}} = \prod_{i=1}^{k}\abs{\mathcal{F}_{P_{l_{i}-1}}} + \sum_{i=1}^{k}\abs{\mathcal{F}_{P_{l_{i}-2}}}\prod_{j\neq i}\abs{\mathcal{F}_{P_{l_{j}-1}}}+\sum_{1\leq i<j\leq k}\abs{\mathcal{F}_{P_{l_{i}}}}\abs{\mathcal{F}_{P_{l_{j}}}},
\]
where $\abs{\mathcal{F}_{P_{-1}}}=1$ and $\abs{\mathcal{F}_{P_{0}}}=0$.
\end{theorem}
\begin{proof}
Let $w_{1},\ldots,w_{k}$ denote the vertices adjacent to $v$ such that, for $i=1,\ldots,k$, $w_{i}$ is in the $(v,u_{i})$ path.
Also, let $F$ denote a minimal fort of $S_{l_{1},\ldots,l_{k}}$.

If $v\notin F$, then $N(v)\cap F=\{w_{i},w_{j}\}$, for some $i\neq j$.
Furthermore, $u_{i},u_{j}\in F$, and Theorem~\ref{thm:tree_induced_path} implies that the $(w_{i},u_{i})$- and $(w_{j},u_{j})$-paths induce a minimal fort on a path of order $l_{i}$ and $l_{j}$, respectively.
Therefore, there are
\[
\sum_{1\leq i<j\leq k}\abs{\mathcal{F}_{P_{l_{i}}}}\abs{\mathcal{F}_{P_{l_{j}}}} 
\]
minimal forts that do not include $v$.

If $v\in F$, then $u_{1},\ldots,u_{k}\in F$ and $\abs{N(v)\cap F}\leq 1$.
Suppose that $\abs{N(v)\cap F}=0$.
Then, for $i=1,\ldots,k$, $w_{i}\notin F$.
Therefore, $l_{i}\geq 2$ since $w_{i}$ is adjacent to exactly one vertex in $F$ other than $v$, which we denote by $w'_{i}$.
Now, Theorem~\ref{thm:tree_induced_path} implies that the $(w'_{i},u_{i})$-path induces a minimal fort on a path of order $l_{i}-1$.
Therefore, there are
\[
\prod_{i=1}^{k}\abs{\mathcal{F}_{P_{l_{i}-1}}}
\]
minimal forts that include $v$ and none of the neighbors of $v$. 

Suppose that $\abs{N(v)\cap F}=1$.
Then, there is exactly one $i\in\{1,\ldots,k\}$ such that $w_{i}\in F$.
Note that $l_{i}\neq 2$.
Indeed, if $l_{i}=2$, then $N(w_{i})=\{v,u_{i}\}$.
Therefore, Theorem~\ref{thm:tree_induced_path} implies that the $(v,u_{i})$-path induces a minimal fort on a path of order $3$.
However, the only such minimal fort does not include $w_{i}$.
If $l_{i}=1$, then $w_{i}=u_{i}$ and there is exactly one minimal fort induced by the $(v,u_{i})$-path. 
If $l_{i}\geq 3$, then $w_{i}$ has exactly one neighbor not in $F$, which we denote by $w'_{i}$.
Let $w''_{i}$ denote the vertex in $N(w'_{i})\cap F\setminus\{w_{i}\}$. 
Then, Theorem~\ref{thm:tree_induced_path} implies that the $(w''_{i},u_{i})$-path induces a minimal fort on a path of order $l_{i}-2$.

Finally, note that for all $j\neq i$, we have $w_{j}\notin F$, which implies that $l_{j}\geq 2$.
In this case, $w_{j}$ is adjacent to exactly one vertex in $F$ other than $v$, which we denote by $w'_{j}$.
Again, Theorem~\ref{thm:tree_induced_path} implies that the $(w'_{j},u_{j})$-path induces a minimal fort on a path of order $l_{j}-1$.
Therefore, there are 
\[
\sum_{i=1}^{k}\abs{\mathcal{F}_{P_{l_{i}-2}}}\prod_{j\neq i}\abs{\mathcal{F}_{P_{l_{j}-1}}}
\]
minimal forts that include $v$ and exactly one neighbor of $v$.
\end{proof}

It is worth noting that Theorem~\ref{thm:spider_graph} implies that the star graph, that is, the spider graph with $(n-1)$ legs of length $1$, has $\binom{n-1}{2}$ minimal forts.
Obviously, this result is also implied by Proposition~\ref{prop:complete_bipartite}. 
The following corollary applies Theorem~\ref{thm:spider_graph} to spiders with $k$ legs of length $l$.
\begin{corollary}\label{cor:symmetric_spider}
Let $S_{l_{1},\ldots,l_{k}}$ denote a spider graph with junction vertex $v$, pendant vertices $u_{1},\ldots,u_{k}$, where $k\geq 2$ and for $i=1,\ldots,k$, the length of the path from $v$ to $u_{i}$ is $l_{i}=l\geq 1$. 
Then, the number of minimal forts satisfies
\[
\abs{\mathcal{F}_{S_{l_{1},\ldots,l_{k}}}} = \abs{\mathcal{F}_{P_{l-1}}}^{k} + k\abs{\mathcal{F}_{P_{l-2}}}\abs{\mathcal{F}_{P_{l-1}}}^{k-1} + \binom{k}{2}\abs{\mathcal{F}_{P_{l}}}^{2},
\]
where $\abs{\mathcal{F}_{P_{-1}}}=1$ and $\abs{\mathcal{F}_{P_{0}}}=0$.
\end{corollary}

\begin{figure}[ht]
\centering
\begin{tikzpicture}[every node/.style={circle,draw=black,inner sep=1.5}, scale=1.1]
    \begin{scope}[xshift=0,yshift=0]
    \foreach \i in {0,1,...,4}
    {
    \draw (0,0)--({\i*360/8}:1); \draw ({\i*360/8}:1)--({\i*360/8}:1.5);
    \draw[loosely dotted, thick] ({\i*360/8}:1.6)--({\i*360/8}:2.5); \draw ({\i*360/8}:2.5)--({\i*360/8}:3);
    }
    \foreach \i in {1,2,3,4}
    {
    \node[fill=white,label={[label distance=.3]{(5-\i)*360/8}:{\scriptsize $u_{{\i}}$}}] at ({(5-\i)*360/8}:3) {};
    }
    \node[fill=white,label=below left:{\scriptsize $w_1$}] at ({4*360/8}:1) {};
    \node[fill=white,label=left:{\scriptsize $w_2$}] at ({3*360/8}:1) {};
    \node[fill=white,label=above right:{\scriptsize $w_3$}] at ({2*360/8}:1) {};
    \node[fill=white,label=right:{\scriptsize $w_4$}] at ({1*360/8}:1) {};
    \node[fill=white,label=below right:{\scriptsize $w_k$}] at (0:1) {};
    \node[fill=white,label=right:{\scriptsize $u_k$}] at (0:3) {};
    
    \node[fill=white, label=below:{\scriptsize $v$}] at (0,0) {};
    \draw[loosely dotted, thick] ({0.25*360/8}:2.5) -- ({0.75*360/8}:2.5);
    \end{scope}
    \begin{scope}[xshift=0,yshift=-32]
    \node[circle,draw=white] at (0,0) {$\downarrow$};
    \end{scope}
    \begin{scope}[xshift=-175,yshift=-50]
    \draw (0,0)--(.4,0);
    \draw[loosely dotted, thick] (.5,0)--(1.1,0);
    \draw (1.1,0)--(3.4,0);
    \draw[loosely dotted, thick] (3.5,0)--(4.1,0);
    \draw (4.1,0)--(5.65,0);
    \draw[loosely dotted, thick] (5.75,0)--(6.35,0);
    \draw (6.35,0)--(7.9,0);
    \draw[loosely dotted, thick] (8,0)--(8.6,0);
    \draw (8.6,0)--(9.4,0);
    \draw[loosely dotted, ultra thick] (9.55,0)--(10.6,0);
    \draw (10.85,0)--(11.65,0);
    \draw[loosely dotted, thick] (11.75,0)--(12.35,0);
    \draw (12.35,0)--(12.75,0);

    \node[fill=white,label=below:{\scriptsize $u_1$}] at (0,0) {};
    \node[fill=white,label=below:{\scriptsize $w_1$}] at (1.5,0) {};
    \node[fill=white,label=below:{\scriptsize $v$}] at (2.25,0) {};
    \node[fill=white,label=below:{\scriptsize $w_2$}] at (3,0) {};
    \node[fill=white,label=below:{\scriptsize $u_2$}] at (4.5,0) {};
    \node[fill=white,label=below:{\scriptsize $w_3$}] at (5.25,0) {};
    \node[fill=white,label=below:{\scriptsize $u_3$}] at (6.75,0) {};
    \node[fill=white,label=below:{\scriptsize $w_4$}] at (7.5,0) {};
    \node[fill=white,label=below:{\scriptsize $u_4$}] at (9,0) {};
    \node[fill=white,label=below:{\scriptsize $w_k$}] at (11.25,0) {};
    \node[fill=white,label=below:{\scriptsize $u_k$}] at (12.75,0) {};
    \end{scope}
\end{tikzpicture}
\caption{Deconstruction of the spider $S_{l_{1},\ldots,l_{k}}$ into a path of order $n=\sum_{i=1}^{k}l_{i}+1$.}
\label{fig:spider_path}
\end{figure}
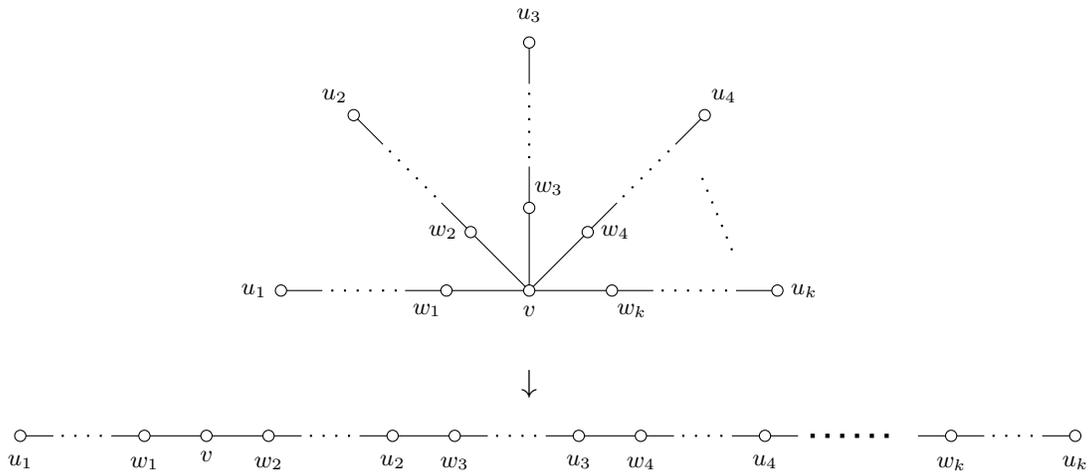

The following theorem provides a bound on the number of minimal forts of a spider graph based on the number of pendant vertices and the number of minimal forts of a path graph of the same order. 
While the bound is clearly not sharp in general, it does imply that the asymptotic growth rate of all spiders is bounded above by the asymptotic growth rate of the path graph, that is, the plastic ratio. 
Moreover, we suspect that the asymptotic growth rate of the path graph is extremal over all trees, see Conjecture~\ref{con:tree_gr}.
\begin{theorem}\label{thm:spider_graph_bound}
Let $S_{l_{1},\ldots,l_{k}}$ denote a spider graph with junction vertex $v$, pendant vertices $u_{1},\ldots,u_{k}$, where $k\geq 2$ and for $i=1,\ldots,k$, the length of the path from $v$ to $u_{i}$ is $l_{i}\geq 1$. 
Then, the number of minimal forts satisfies
\[
\abs{\mathcal{F}_{S_{l_{1},\ldots,l_{k}}}} \leq \binom{k}{2}\abs{\mathcal{F}_{P_{n}}},
\]
where $n=\sum_{i=1}^{k}l_{i}+1$ is the order of $S_{l_{1},\ldots,l_{k}}$. 
\end{theorem}
\begin{proof}
Let $w_{1},\ldots,w_{k}$ denote the vertices adjacent to $v$ such that, for $i=1,\ldots,k$, $w_{i}$ is in the $(v,u_{i})$-path.
Also, let $P_{n}$ denote the path graph of order $n$ constructed from the spider graph as follows: For $i=3,\ldots,k$, delete the edge $\{w_{i},v\}$ and add the edge $\{w_{i},u_{i-1}\}$, see Figure~\ref{fig:spider_path}. 
Finally, let $F$ denote a minimal fort of the spider graph, which we denote by  $S = S_{l_{1},\ldots,l_{k}}$ for brevity. 

If $v\notin F$, then $N_{S}(v)\cap F=\{w_{i},w_{j}\}$ and $u_{i},u_{j}\in F$, for some $i\neq j$.
Let $P_{l_{i}+l_{j}+1}$ denote the $(u_{i},u_{j})$-path of $S$ and note that $v$ is a vertex of $P_{l_{i}+l_{j}+1}$. 
Also, let $\mathcal{F}'_{P_{l_{i}+l_{j}+1}}$ denote the minimal forts of $P_{l_{i}+l_{j}+1}$ that don't include $v$.
Then, Theorem~\ref{thm:tree_induced_path} implies that there are
\[
\sum_{1\leq i<j\leq k}\abs{\mathcal{F}'_{P_{l_{i}+l_{j}+1}}}
\]
minimal forts that don't include $v$. 
Since $l_{i}+l_{j}+1\leq n$, for any $i\neq j$, it follows that
\[
\sum_{1\leq i<j\leq k}\abs{\mathcal{F}'_{P_{l_{i}+l_{j}+1}}} \leq \sum_{1\leq i<j\leq k}\abs{\mathcal{F}'_{P_{n}}} = \binom{k}{2}\abs{\mathcal{F}'_{P_{n}}},
\]
where $\mathcal{F}'_{P_{n}}$ denote the minimal forts of $P_{n}$ that don't include $v$.

If $v\in F$, then $u_{1},\ldots,u_{k}\in F$ and $\abs{N_{S}(v)\cap F}\leq 1$.
Furthermore, $F$ is a fort on $P_{n}$. 
Indeed, the pendant vertices of $P_{n}$, $u_{1}$ and $u_{k}$, are in $F$, and there are no two vertices not in $F$ that are adjacent in $P_{n}$.
If $\abs{N_{S}(v)\cap F}=0$, then $F$ is a minimal fort on $P_{n}$ since every vertex in $F$ is adjacent to at most one other vertex from $F$ in $P_{n}$.
If $\abs{N_{S}(v)\cap F}=1$, then there is exactly one $i\in\{1,\ldots,k\}$ such that $w_{i}\in F$.
If $i\in\{1,2\}$, then $F$ is a minimal fort in $P_{n}$.
For $3\leq i\leq k$, if the vertex adjacent to $u_{i-1}$ is also in $F$, then $\hat{F} = F\setminus\{u_{i-1}\}$ is a minimal fort in $P_{n}$.

Therefore, every minimal fort of $S$ that includes $v$ corresponds to a minimal fort of $P_{n}$ that also includes $v$.
Moreover, no two minimal forts of $S$ correspond to the same minimal fort of $P_{n}$.
Hence, there are at most $\abs{\hat{\mathcal{F}}_{P_{n}}}$ minimal forts of $S$ that include $v$, where $\hat{\mathcal{F}}_{P_{n}}$ denotes the minimal forts of $P_{n}$ that include $v$. 
Hence,
\begin{align*}
\abs{\mathcal{F}_{S_{l_{1},\ldots,l_{k}}}} &\leq \binom{k}{2}\abs{\mathcal{F}'_{P_{n}}} + \abs{\hat{\mathcal{F}}_{P_{n}}} \\
&\leq \binom{k}{2}\abs{\mathcal{F}'_{P_{n}}} + \binom{k}{2}\abs{\hat{\mathcal{F}}_{P_{n}}} \\
&= \binom{k}{2}\abs{\mathcal{F}_{P_{n}}}.
\end{align*}
\end{proof}
\section{Graph Products}\label{sec:graph_prod}
In this section, we develop methods for constructing minimal forts of a graph product using minimal forts of the original graphs.
In the process, we provide constructions for additional families of graphs where the number of minimal forts grows exponentially in the order of the graph. 
Throughout, we let $G=(V,E)$ and $G'=(V',E')$ denote graphs in $\mathbb{G}$ such that $V\cap V'=\emptyset$.
\subsection{Join}\label{subsec:join}
The \emph{join} of $G$ and $G'$, denoted $G\vee G'$, is the graph with vertex set $V\left(G\vee G'\right) = V\, \cup\, V'$ and edge set
\[
E\left(G\vee G'\right) = E\cup E'\cup\left\{\{u,u'\}\colon u\in V,\, u'\in V'\right\}.
\]
\begin{proposition}\label{prop:fort_join}
Let $G=(V,E)$ and $G'=(V',E')$ denote graphs in $\mathbb{G}$ such that $V\cap V'=\emptyset$.
If $G$ is connected and $\abs{V}\geq 2$, then every minimal fort of $G$ is a minimal fort of $G\vee G'$.
\end{proposition}
\begin{proof}
Let $G=(V,E)$ be connected and $\abs{V}\geq 2$.
Also, let $F$ denote a minimal fort of $G$. 
Since $G$ is connected and $\abs{V}\geq 2$, it follows that $\abs{F}\geq 2$.
Therefore, every vertex in $G'$ is adjacent to at least two vertices in $F$; hence, $F$ is a fort of $G\vee G'$.
Since no vertex from $G'$ has been added to the fort $F$, the minimality of $F$ in $G$ implies the minimality of $F$ in $G\vee G'$. 
\end{proof}

Since the join operation is commutative, it follows from Proposition~\ref{prop:fort_join} that if $G'$ is connected and of order at least $2$, then every minimal fort of $G'$ is a minimal fort of $G\vee G'$.
The following corollary follows immediately from Proposition~\ref{prop:fort_join} and Corollary~\ref{cor:cycle_graph_explicit}. 
\begin{corollary}\label{cor:wheel}
The wheel graph $C_{n-1}\vee K_{1}$ has at least $\nint{\psi^{n-1}}$ minimal forts for $n\geq 11$.
\end{corollary}

It is worth noting that there are minimal forts for the wheel graph not described by Proposition~\ref{prop:fort_join}; for example, see Figure~\ref{fig:c5_join_k1}.
Specifically, every maximal failed zero forcing set of the wheel graph $C_{n-1}\vee K_{1}$ either contains the vertex from $K_{1}$ and a maximal independent set from $C_{n-1}$, or only contains a maximal collection of vertices from $C_{n-1}$ that is not an independent set and does not induce a $P_{3}$ subgraph.
Beginning with $n=4$, the sequence corresponding to the number of minimal forts of $C_{n-1}\vee K_1$ that contain the vertex from $K_{1}$, equivalently, maximal failed zero forcing sets that do not contain the vertex from $K_{1}$, is
\[
3,4,5,3,14,12,21,25,44,55,78,112,158,\ldots.
\]
As of now, this sequence is not included in the OEIS~\cite{OEIS}, and we are unaware of an explicit formula for the $n$th term in the sequence. 
For SageMath code to generate this sequence, see~\cite{HanelySage}.

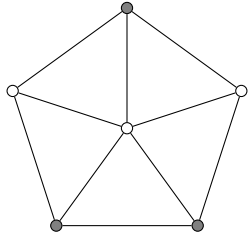
\begin{figure}[ht]
\centering
\begin{tikzpicture}[scale=.4,every node/.style={circle,draw=black,inner sep=1.5}]
    \node[fill=gray] (1) at (0,4) {};
    \node[fill=white] (2) at (-3.80,1.24) {};
    \node[fill=gray] (3) at (-2.35,-3.24) {};
    \node[fill=gray] (4) at (2.35,-3.24) {};
    \node[fill=white] (5) at (3.80,1.24) {};
    \node[fill=white] (6) at (0,0) {};

    \foreach \x/\y in {1/2,2/3,3/4,4/5,5/1,1/6,2/6,3/6,4/6,5/6}
        \draw[black] (\x) -- (\y);
\end{tikzpicture}
\caption{Minimal fort (white) and maximal failed zero forcing set (gray) of $C_{5}\vee K_{1}$ not described by Proposition~\ref{prop:fort_join}.}
\label{fig:c5_join_k1}
\end{figure}

\subsection{Corona Product}\label{subsec:corona}
The \emph{corona} of $G=(V,E)$ with $G'=(V',E')$, denoted $G\circ G'$, is the graph obtained from the disjoint union of $G$ and $\abs{V}$ copies of $G'$ by joining each vertex $u\in V$ with the $u$-copy of $G'$.
Throughout this section, we denote the vertices in the $u$-copy of $G'$ by $V'_{u}=\{v'_{u}\colon v'\in V'\}$.

If $G'$ is connected and of order at least $2$, then Proposition~\ref{prop:fort_join} implies that every minimal fort of $G'$ corresponds to a minimal fort of $G\circ G'$, where the vertices can be selected from any single copy of $G'$. 
We state and prove this observation in the following proposition. 
\begin{proposition}\label{prop:fort_corona_G'}
Let $G=(V,E)$ and $G'=(V',E')$ denote graphs in $\mathbb{G}$ such that $V\cap V'=\emptyset$.
Suppose that $\abs{V'}\geq 2$.
Also, let $F'$ denote a minimal fort of $G'$ and let $u\in V$.
Then, $\hat{F}=\{v'_{u}\colon v'\in F'\}\subseteq V'_{u}$ is a minimal fort of $G\circ G'$. 
\end{proposition}
\begin{proof}
By Proposition~\ref{prop:fort_join}, $\hat{F}$ is a minimal fort of the induced subgraph of $G\circ G'$ obtained from the vertices in $V'_{u}\cup\{u\}$. 
Since no vertex in $V(G\circ G')\setminus\left(V'_{u}\cup\{u\}\right)$ is adjacent to any vertex in $\hat{F}$, it follows that $\hat{F}$ is a minimal fort of $G\circ G'$.
\end{proof}

Under certain cases, we can use the minimal forts of $G$ to construct minimal forts of the corona product $G\circ G'$, as shown in the following proposition. 
\begin{proposition}\label{prop:fort_corona_Cn_Kr}
Let $C_{n}$ denote a cycle of order $n$ and $K_{r}$ denote a complete graph of order $r\geq 1$.
Let $C_{n}\circ K_{r}$ denote the corona of $C_{n}$ with $K_{r}$ formed by joining each vertex of $C_{n}$ with a copy of $K_{r}$.
Further, let $F$ denote a minimal fort of $C_{n}$ and for each $u\in F$ select a single vertex $v_{u}$ from the $u$-copy of $K_{r}$.
Then, $\hat{F}=F\,\cup\,\{v_{u}\colon u\in F\}$ is a minimal fort of $C_{n}\circ K_{r}$.
\end{proposition}
\begin{proof}
Note that every vertex in a copy of $K_{r}$ is either in $\hat{F}$ or has exactly two neighbors in $\hat{F}$. 
Furthermore, every vertex $v\in V(C_{n})$ such that $v\notin\hat{F}$ satisfies
\[
N_{C_{n}\circ K_{r}}(v)\cap\hat{F} = N_{C_{n}}(v)\cap F.
\]
Hence, $\hat{F}$ is a fort of $C_{n}\circ K_{r}$.
Now, we must show that $\hat{F}$ is minimal.
To this end, we show that $\hat{S}=V(C_{n}\circ K_{r})\setminus\hat{F}$ is a maximal failed zero forcing set of $C_{n}\circ K_{r}$. 

Since $F$ is a minimal fort of $C_{n}$, Theorem~\ref{thm:failed_zf} implies that 
\[
S=V(C_{n})\cap\hat{S}=V(C_{n})\setminus{F}
\]
is a maximal failed zero forcing set of $C_{n}$.
Thus, for each $u\in F$, $S\cup\{u\}$ is a zero forcing set of $C_{n}$.
Furthermore, there exists a collection of forces where each vertex of $F\setminus\{u\}$ is forced by a vertex from $S$.
This same collection of forcings can be applied in $C_{n}\circ K_{r}$ since for each $v\in S$ the $v$-copy of $K_{r}$ does not intersect $\hat{F}$. 
Therefore, for any $u\in F$, $\hat{S}\cup\{u\}$ is a zero-forcing set of $C_{n}\circ K_{r}$.

Finally, let $u\in F$ and let $v_{u}\in\hat{F}$ be the corresponding vertex from the $u$-copy of $K_{r}$.
Then, $\hat{S}\cup\{v_{u}\}$ is a zero forcing set of $C_{n}\circ K_{r}$ since $v_{u}$ can force $u$ at which point we have a zero forcing set by the previous argument. 
\end{proof}

Note that Proposition~\ref{prop:fort_corona_Cn_Kr} does not hold when replacing $C_{n}$ by a general graph $G\in\mathbb{G}$.
Indeed, consider the \emph{ladder graph} $l_{4}=K_{2}\Box P_{4}$ and its minimal fort depicted on the left of Figure~\ref{fig:ladder4_corona_k1}. 
Extending this minimal fort of $l_{4}$ to a fort of $l_{4}\circ K_{1}$ via the construction in Proposition~\ref{prop:fort_corona_Cn_Kr} would yield the fort in the center of Figure~\ref{fig:ladder4_corona_k1}.
However, this fort is not minimal as the rightmost depiction of Figure~\ref{fig:ladder4_corona_k1} demonstrates. 

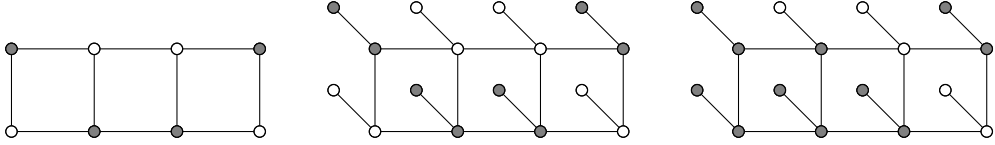
\begin{figure}[ht]
\centering
\begin{tikzpicture}[every node/.style={circle,draw=black,inner sep=1.5},scale=1.1]
\begin{scope}
    \draw[black] (0,0)--(3,0)--(3,1)--(0,1)--cycle;
    \draw[black] (1,0)--(1,1);
    \draw[black] (2,0)--(2,1);
    \foreach \x in {0,1,...,3}
    \foreach \y in {0,1}
        {
            \draw[fill=white] (\x,\y) circle (2pt) node {};
        }
    \node[fill=gray] at (1,0) {};
    \node[fill=gray] at (2,0) {};
    \node[fill=gray] at (0,1) {};
    \node[fill=gray] at (3,1) {};
\end{scope}
\begin{scope}[xshift=125]
    \draw[black] (0,0)--(3,0)--(3,1)--(0,1)--cycle;
    \foreach \j in {1,2}
    {
        \draw[black] (\j,0)--(\j,1);
    }
    \foreach \i in {0,1,2,3}
    {
        \draw[black] (\i,0)--(\i-0.5,0.5);
        \draw[black] (\i,1)--(\i-0.5,1.5);
    }
    \foreach \x in {0,1,...,3}
    \foreach \y in {0,1}
        {
            \draw[fill=white] (\x,\y) circle (2pt) node {};
            \draw[fill=white] (\x-.5,\y+.5) circle (2pt) node {};
        }
    \node[fill=gray] at (0,1) {};
    \node[fill=gray] at (1,0) {};
    \node[fill=gray] at (2,0) {};
    \node[fill=gray] at (3,1) {};
    \node[fill=gray] at (-.5,1.5) {};
    \node[fill=gray] at (0.5,0.5) {};
    \node[fill=gray] at (1.5,0.5) {};
    \node[fill=gray] at (2.5,1.5) {};
\end{scope}
\begin{scope}[xshift=250]
    \draw[black] (0,0)--(3,0)--(3,1)--(0,1)--cycle;
    \foreach \j in {1,2}
    {
        \draw[black] (\j,0)--(\j,1);
    }
    \foreach \i in {0,1,2,3}
    {
        \draw[black] (\i,0)--(\i-0.5,0.5);
        \draw[black] (\i,1)--(\i-0.5,1.5);
    }
    \foreach \x in {0,1,...,3}
    \foreach \y in {0,1}
        {
            \draw[fill=white] (\x,\y) circle (2pt) node {};
            \draw[fill=white] (\x-.5,\y+.5) circle (2pt) node {};
        }
    \node[fill=gray] at (0,1) {};
    \node[fill=gray] at (1,0) {};
    \node[fill=gray] at (2,0) {};
    \node[fill=gray] at (3,1) {};
    \node[fill=gray] at (0,0) {};
    \node[fill=gray] at (1,1) {};
    \node[fill=gray] at (-.5,.5) {};
    \node[fill=gray] at (-.5,1.5) {};
    \node[fill=gray] at (0.5,0.5) {};
    \node[fill=gray] at (1.5,0.5) {};
    \node[fill=gray] at (2.5,1.5) {};
\end{scope}
\end{tikzpicture}
\caption{A minimal fort (white) and maximal failed zero forcing set (gray) of the ladder graph $l_{4}=K_{2}\Box P_{4}$ (left), the corresponding extension to a fort of $l_{4}\circ K_{1}$ (center), and a minimal fort (white) and maximal failed zero forcing set (gray) of $l_{4}\circ K_1$ (right).}
\label{fig:ladder4_corona_k1}
\end{figure}

The following corollary follows immediately from Proposition~\ref{prop:fort_corona_Cn_Kr} and Corollary~\ref{cor:cycle_graph_explicit}. 
\begin{corollary}\label{cor:sunlet}
The sunlet graph $C_{n}\circ K_{1}$ has at least $\nint{\psi^{n}}$ minimal forts for $n\geq 10$.
\end{corollary}

We note that there are minimal forts for the sunlet graph not described by Proposition~\ref{prop:fort_corona_Cn_Kr}; for example, see Figure~\ref{fig:c5_corona_k1}.

\begin{figure}[ht]
\centering
\begin{tikzpicture}[scale=.35,every node/.style={circle,draw=black,inner sep=1.5}]
\begin{scope}
    \node[fill=white] (1) at (0,4) {};
    \node[fill=gray] (2) at (-3.80,1.24) {};
    \node[fill=gray] (3) at (-2.35,-3.24) {};
    \node[fill=gray] (4) at (2.35,-3.24) {};
    \node[fill=gray] (5) at (3.80,1.24) {};

    \node[fill=white] (6) at (0,5) {};
    \node[fill=white] (7) at (-4.51,1.95) {};
    \node[fill=gray] (8) at (-3.06,-3.95) {};
    \node[fill=gray] (9) at (3.06,-3.95) {};
    \node[fill=white] (10) at (4.51,1.95) {};

    \foreach \x/\y in {1/2,2/3,3/4,4/5,5/1,1/6,2/7,3/8,4/9,5/10}
        \draw[black] (\x) -- (\y);
\end{scope}
\begin{scope}[xshift=400]
    \node[fill=gray] (1) at (0,4) {};
    \node[fill=white] (2) at (-3.80,1.24) {};
    \node[fill=gray] (3) at (-2.35,-3.24) {};
    \node[fill=gray] (4) at (2.35,-3.24) {};
    \node[fill=white] (5) at (3.80,1.24) {};

    \node[fill=gray] (6) at (0,5) {};
    \node[fill=white] (7) at (-4.51,1.95) {};
    \node[fill=white] (8) at (-3.06,-3.95) {};
    \node[fill=white] (9) at (3.06,-3.95) {};
    \node[fill=white] (10) at (4.51,1.95) {};

    \foreach \x/\y in {1/2,2/3,3/4,4/5,5/1,1/6,2/7,3/8,4/9,5/10}
        \draw[black] (\x) -- (\y);
\end{scope}
\end{tikzpicture}
\caption{Minimal forts (white) and maximal failed zero forcing sets (gray) of $C_{5}\circ K_{1}$ not described by Proposition~\ref{prop:fort_corona_Cn_Kr}.}
\label{fig:c5_corona_k1}
\end{figure}
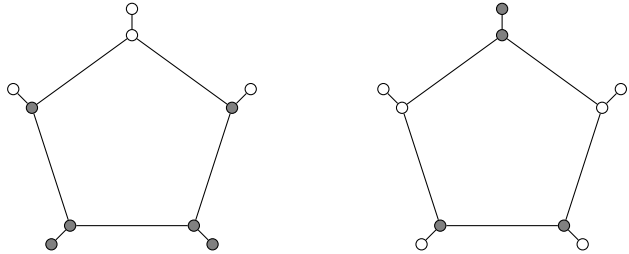

\subsection{Vertex Sum}\label{subsec:vert_sum}
The \emph{vertex sum} of $G$ and $G'$, denoted $G\oplus_{v}G'$, is the graph formed by identifying vertices $u\in V$ and $u'\in V'$ to form the vertex $v$.
In particular, $G\oplus_{v}G'$ has vertex set
\[
V(G\oplus_{v}G') = \left(V\setminus\{u\}\right)\cup\left(V'\setminus\{u'\}\right)\cup\{v\}
\]
and edge set
\[
E(G\oplus_{v}G') = \left(E\setminus{S_{u}}\right)\cup\left(E'\setminus{S_{u'}}\right)\cup S_{v},
\]
where $S_{u}$ denotes all edges of $G$ that are incident on $u$, $S_{u'}$ denotes all edges of $G'$ that are incident on $u'$, and $S_{v}=\left\{\{v,w\}\colon w\in N_{G}(u)\cup N_{G'}(u')\right\}$.

The following theorem shows how to combine maximal failed zero forcing sets of $G$ and $G'$ to create failed zero forcing sets of $G\oplus_{v}G'$.
Furthermore, we provide a characterization of when the resulting failed zero forcing sets are maximal. 
\begin{theorem}\label{thm:failed_zf_sum}
Let $G=(V,E)$ and $G'=(V',E')$ denote graphs in $\mathbb{G}$ such that $V\cap V'=\emptyset$. 
Let $G\oplus_{v}G'$ denote the vertex sum of $G$ and $G'$ formed by identifying vertices $u\in V$ and $u'\in V'$ to form the vertex $v$.
Suppose that $S\subseteq V\setminus\{u\}$ and $S'\subseteq V'\setminus\{u'\}$ are maximal failed zero forcing sets of $G$ and $G'$, respectively. 
Then, $S\cup S'$ is a failed zero forcing set of $G\oplus_{v}G'$.
Moreover, $S\cup S'$ is a maximal failed zero forcing set of $G\oplus_{v}G'$ if and only if $S\cup S'\cup\{v\}$ is a zero forcing set. 
\end{theorem}
\begin{proof}
It is clear that $S\cup S'$ is a failed zero forcing set of $G\oplus_{v}G'$.
All that remains is to show that $S\cup S'$ is maximal if and only if $S\cup S'\cup\{v\}$ is a zero forcing set of $G\oplus_{v}G'$.

To that end, suppose that $S\cup S'$ is maximal.
Then, adding any vertex to $S\cup S'$ results in a zero forcing set; hence, $S\cup S'\cup\{v\}$ is a zero forcing set of $G\oplus_{v}G'$.

Conversely, suppose that $S\cup S'\cup\{v\}$ is a zero forcing set of $G\oplus_{v}G'$.
Since $S$ is a maximal failed zero forcing set of $G$, it follows that $S\cup A$ is a zero forcing set for any non-empty $A\subseteq V\setminus{S}$.
Furthermore, $S\cup S'\cup A$ is a zero forcing set of $G\oplus_{v}G'$ since once $u$ is forced in $G$ we have $S\cup S'\cup\{v\}$ forced in $G\oplus_{v}G'$ at which point all vertices can be forced.
A similar argument shows that $S\cup S'\cup A'$ is a zero forcing set of $G\oplus_{v}G'$ for any non-empty $A'\subseteq V'\setminus{S'}$.
Therefore, $S\cup S'$ is a maximal failed zero forcing set of $G\oplus_{v}G'$.
\end{proof}

Next, we show how to combine minimal forts of $G$ and $G'$ to create forts of $G\oplus_{v}G'$.
Furthermore, we provide a characterization of when the resulting forts are minimal. 
\begin{corollary}\label{cor:fort_sum}
Let $G=(V,E)$ and $G'=(V',E')$ denote graphs in $\mathbb{G}$ such that $V\cap V'=\emptyset$. 
Let $G\oplus_{v}G'$ denote the vertex sum of $G$ and $G'$ formed by identifying vertices $u\in V$ and $u'\in V'$ to form the vertex $v$.
Suppose that $F\subseteq V$ and $F'\subseteq V'$ are minimal forts of $G$ and $G'$, respectively, such that $u\in F$ and $u'\in F'$.
Then, 
\[
\hat{F} = \left(F\setminus\{u\}\right)\cup\left(F'\setminus\{u'\}\right)\cup\{v\}
\]
is a fort of $G\oplus_{v}G'$.
Moreover, $\hat{F}$ is a minimal fort if and only if $\hat{F}\setminus\{v\}$ does not contain a fort of $G\oplus_{v}G'$.
\end{corollary}
\begin{proof}
It is clear that $\hat{F}$ is a fort of $G\oplus_{v}G'$.
All that remains is to show that $\hat{F}$ is minimal if and only if $\hat{F}\setminus\{v\}$ is not a fort.

To that end, let $S=V\setminus{F}$ and $S'=V'\setminus{F'}$.
Then, by Theorem~\ref{thm:failed_zf}, $\hat{F}$ is a minimal fort of $G\oplus_{v}G'$ if and only if $S\cup S'$ is a maximal failed zero forcing set, which, by Theorem~\ref{thm:failed_zf_sum}, is true if and only if $S\cup S'\cup\{v\}$ is a zero forcing set. 
Hence, all we need to show is that $\hat{F}\setminus\{v\}$ does not contain a fort if and only if $S\cup S'\cup\{v\}$ is a zero forcing set.

Suppose that $S\cup S'\cup\{v\}$ is a zero forcing set of $G\oplus_{v}G'$.
Then, by Theorem~\ref{thm:fort-cover}, $S\cup S'\cup\{v\}$ must intersect every fort of $G\oplus_{v}G'$.
Since $\hat{F}\setminus\{v\}$ is the complement of $S\cup S'\cup\{v\}$, it follows
that $\hat{F}\setminus\{v\}$ does not contain a fort. 

Conversely, suppose that $S\cup S'\cup\{v\}$ is not a zero forcing set of $G\oplus_{v}G'$.
Then, by Theorem~\ref{thm:fort-cover}, $S\cup S'\cup\{v\}$ does not intersect every fort of $G\oplus_{v}G'$.
Since $\hat{F}\setminus\{v\}$ is the complement of $S\cup S'\cup\{v\}$, it follows that $\hat{F}\setminus\{v\}$ contains a fort. 
\end{proof}

Before proceeding, we note that every maximal failed zero forcing set of $G\oplus_{v}G'$ that does not include $v$ can be constructed as in Theorem~\ref{thm:failed_zf_sum}.
Equivalently, every minimal fort of $G\oplus_{v}G'$ that contains $v$ can be constructed as in Corollary~\ref{cor:fort_sum}.
For reference, we state and prove this observation below. 
\begin{proposition}\label{prop:exact}
Let $G=(V,E)$ and $G'=(V',E')$ denote graphs in $\mathbb{G}$ such that $V\cap V'=\emptyset$. 
Let $G\oplus_{v}G'$ denote the vertex sum of $G$ and $G'$ formed by identifying vertices $u\in V$ and $u'\in V'$ to form the vertex $v$.
Suppose that $\hat{S}$ is a maximal failed zero forcing set of $G\oplus_{v}G'$ such that $v\notin\hat{S}$.
Then, $\hat{F}=V(G\oplus_{v}G')\setminus{\hat{S}}$ is a minimal fort of $G\oplus_{v}G'$ such that $v\in\hat{F}$.
Moreover, $S=V\cap\hat{S}$ and $S'=V'\cap\hat{S}$ are maximal failed zero forcing sets of $G$ and $G'$, respectively, such that $\hat{S}=S\cup S'$.
Equivalently, $F=V\setminus{S}$ and $F'=V'\setminus{S'}$ are minimal forts of $G$ and $G'$, respectively, such that $\hat{F}=\left(F\setminus\{u\}\right)\cup\left(F'\setminus\{u'\}\right)\cup\{v\}$.
\end{proposition}
\begin{proof}
We prove the statement about maximal failed zero forcing sets and note that the equivalent statement regarding minimal forts follows from Theorem~\ref{thm:failed_zf}.
To this end, let $\hat{S}$ denote a maximal failed zero forcing set of $G\oplus_{v}G'$ such that $v\notin\hat{S}$ and define $S=V\cap\hat{S}$ and $S'=V'\cap\hat{S}$.
Clearly, $S$ and $S'$ are failed zero forcing sets of $G$ and $G'$, respectively, since if any vertex in either set could force that would contradict $\hat{S}$ being a maximal failed zero forcing set of $G\oplus_{v}G'$.

Moreover, since $\hat{S}$ is a maximal failed zero forcing set of $G\oplus_{v}G'$, we know that $\hat{S}\cup\{v\}$ is a zero forcing set. 
Since no vertex in $V\setminus\{u\}$ is forced by a vertex in $V'\setminus\{u'\}$, it follows that $S\cup\{u\}$ is a zero forcing set of $G$.
In addition, $S\cup A$  is a zero forcing set of $G$ for any non-empty $A\subset V\setminus{S}$.
If $u\in A$, then this statement follows from the fact that $S\cup\{u\}$ is a zero forcing set of $G$. 
Suppose then that $u\notin A$.
Since $\hat{S}\cup A$ is a zero forcing set of $G\oplus_{v}G'$, it follows that the vertices in $S\cup A$ will eventually force $u$ at which point we have a zero forcing set of $G$.

Therefore, $S$ is a maximal failed zero forcing set of $G$.
A similar argument shows that $S'$ is a maximal failed zero forcing set of $G'$.
\end{proof}

The \emph{windmill graph}, denoted $Wd(r,k)$ for $r\geq 3$ and $k\geq 2$, is constructed via the vertex sum of $k$ disjoint copies of $K_{r}$, where a vertex from each copy of $K_{r}$ is identified to form the vertex $v$. 
In particular,
\begin{align*}
Wd(r,2) &= K_{r}\oplus_{v}K_{r} \\
Wd(r,k) &= Wd(r,k-1)\oplus_{v}K_{r},~k\geq 3. \\
\end{align*}

Note that the vertex $v$ in the graph $Wd(r,k)$ is a cut vertex since its deletion results in $k$ connected components.
The following corollary shows that $Wd(r,k)$ has an exponential number of minimal forts in the order of the graph that contain the cut vertex $v$.
Hence, the windmill graph is another example of a family of graphs for which the bound in Theorem~\ref{thm:zf_cart_bound} is not very good. 
\begin{corollary}\label{cor:windmill}
The windmill graph $Wd(r,k)$ has exactly $(r-1)^{k}$ minimal forts that contain the cut vertex $v$ and exactly $k\binom{r-1}{2}$ minimal forts that do not contain $v$. 
\end{corollary}
\begin{proof}
We proceed via induction on $k$.
Let $k=2$ and let $K_{r}$ and $K'_{r}$ denote disjoint copies of the complete graph of order $r$.
Also, let $Wd(r,2) = K_{r}\oplus_{v}K'_{r}$ denote the vertex sum of $K_{r}$ and $K'_{r}$ formed by identifying vertices $u\in V(K_{r})$ and $u'\in V(K'_{r})$ to form the vertex $v$.
Now, let $F\subset V(K_{r})$ and $F'\subset V(K'_{r})$ denote two element subsets such that $u\in F$ and $u'\in F'$.
Then, $F$ and $F'$ are minimal forts of $K_{r}$ and $K'_{r}$; moreover, Corollary~\ref{cor:fort_sum} implies that
\[
\hat{F} = \left(F\setminus\{u\}\right)\cup\left(F'\setminus\{u'\}\right)\cup\{v\}
\]
is a minimal fort of $Wd(r,2)$ since $\hat{F}\setminus\{v\}$ does not contain a fort as it only contains a single vertex from each copy of $K_{r}$.
Note that there are $(r-1)$ pairs of vertices in $K_{r}$ and $K'_{r}$ that include $u$ and $u'$, respectively.
Hence, there are at least $(r-1)^{2}$ minimal forts of $Wd(r,2)$ that include the vertex $v$.
Moreover, by Proposition~\ref{prop:exact}, these are the only minimal forts of $Wd(r,2)$ that contain the vertex $v$.
The minimal forts that do not include the vertex $v$ must contain $2$ vertices from $V(K_{r})\setminus\{u\}$ or two vertices from $V(K'_{r})\setminus\{u'\}$.
Therefore, there are $2\binom{r-1}{2}$ minimal forts of $Wd(r,2)$ that do not contain $v$.  

Now, let $k\geq 2$ and suppose that $Wd(r,k)$ has exactly $(r-1)^{k}$ minimal forts that contain the vertex $v$ and $k\binom{r-1}{2}$ minimal forts that do not contain $v$. 
Further, let $Wd(r,k+1)=Wd(r,k)\oplus_{v} K_{r}$ denote the vertex sum of $Wd(r,k)$ and $K_{r}$ formed by identifying vertices $v\in V(Wd(r,k))$ and $u\in V(K_{r})$ to form the vertex $v$. 
Let $F\subset V(Wd(r,k))$ be a minimal fort of $Wd(r,k)$ that contains $v$ and let $F'\subset V(K_{r})$ be a minimal fort of $K_{r}$ that contains $u$.
Then, Corollary~\ref{cor:fort_sum} implies that
\[
\hat{F} = \left(F\setminus\{v\}\right)\cup\left(F'\setminus\{u\}\right)\cup\{v\}
\]
is a minimal fort of $Wd(r,k+1)$ since $\hat{F}\setminus\{v\}$ does not contain a fort as it only contains a single vertex from each copy of $K_{r}$.
Note that there are $(r-1)^{k}$ minimal forts of $Wd(r,k)$ that contain the vertex $v$ and there are $(r-1)$ minimal forts of $K_{r}$ that contain the vertex $u$.
Hence, there are at least $(r-1)^{k+1}$ minimal forts of $Wd(r,k+1)$ that contain the vertex $v$.
Moreover, by Proposition~\ref{prop:exact}, these are the only minimal forts of $Wd(r,k+1)$ that contain the vertex $v$.
The minimal forts that do not include the vertex $v$ are made up of a minimal fort of $Wd(r,k)$ that does not include the vertex $v$ or a minimal fort of $K_{r}$ that does not include the vertex $u$.
Therefore, there are $k\binom{r-1}{2}$ + $\binom{r-1}{2} = (k+1)\binom{r-1}{2}$ minimal forts of $Wd(r,k+1)$ that do not contain $v$. 
\end{proof}

Observe that Corollary~\ref{cor:windmill} gives a complete description of all the minimal forts of the windmill graph $Wd(r,k)$.
For example, an enumeration of the minimal forts for $Wd(3,3)$ is provided in Figure~\ref{fig:forts_Wd(3,3)}.

\begin{figure}[ht]
    \centering
    \begin{tikzpicture}[every node/.style={circle,draw=black,inner sep=1.5},scale=1.1]
    \begin{scope}
    \draw[black] (-1,0)--(1,0)--(.5,{sqrt(3)/2})--(0,0)--(-.5,{sqrt(3)/2})--cycle;
    \draw[black] (0,0)--(-.5,{-sqrt(3)/2})--(.5,{-sqrt(3)/2})--cycle;
    \node[fill=white] at (-1,0) {};
    \node[fill=white] at (0,0) {};
    \node[fill=gray] at (1,0) {};
    \node[fill=white] at (0.5,{sqrt(3)/2}) {};
    \node[fill=gray] at (-0.5,{sqrt(3)/2}) {};
    \node[fill=gray] at (-0.5,{-sqrt(3)/2}) {};
    \node[fill=white] at (0.5,{-sqrt(3)/2}) {};
    \end{scope}
    
    \begin{scope}[xshift=75]
    \draw[black] (-1,0)--(1,0)--(.5,{sqrt(3)/2})--(0,0)--(-.5,{sqrt(3)/2})--cycle;
    \draw[black] (0,0)--(-.5,{-sqrt(3)/2})--(.5,{-sqrt(3)/2})--cycle;
    \node[fill=white] at (-1,0) {};
    \node[fill=white] at (0,0) {};
    \node[fill=gray] at (1,0) {};
    \node[fill=white] at (0.5,{sqrt(3)/2}) {};
    \node[fill=gray] at (-0.5,{sqrt(3)/2}) {};
    \node[fill=white] at (-0.5,{-sqrt(3)/2}) {};
    \node[fill=gray] at (0.5,{-sqrt(3)/2}) {};
    \end{scope}

    \begin{scope}[xshift=150]
    \draw[black] (-1,0)--(1,0)--(.5,{sqrt(3)/2})--(0,0)--(-.5,{sqrt(3)/2})--cycle;
    \draw[black] (0,0)--(-.5,{-sqrt(3)/2})--(.5,{-sqrt(3)/2})--cycle;
    \node[fill=white] at (-1,0) {};
    \node[fill=white] at (0,0) {};
    \node[fill=white] at (1,0) {};
    \node[fill=gray] at (0.5,{sqrt(3)/2}) {};
    \node[fill=gray] at (-0.5,{sqrt(3)/2}) {};
    \node[fill=gray] at (-0.5,{-sqrt(3)/2}) {};
    \node[fill=white] at (0.5,{-sqrt(3)/2}) {};
    \end{scope}

    \begin{scope}[xshift=225]
    \draw[black] (-1,0)--(1,0)--(.5,{sqrt(3)/2})--(0,0)--(-.5,{sqrt(3)/2})--cycle;
    \draw[black] (0,0)--(-.5,{-sqrt(3)/2})--(.5,{-sqrt(3)/2})--cycle;
    \node[fill=white] at (-1,0) {};
    \node[fill=white] at (0,0) {};
    \node[fill=white] at (1,0) {};
    \node[fill=gray] at (0.5,{sqrt(3)/2}) {};
    \node[fill=gray] at (-0.5,{sqrt(3)/2}) {};
    \node[fill=white] at (-0.5,{-sqrt(3)/2}) {};
    \node[fill=gray] at (0.5,{-sqrt(3)/2}) {};
    \end{scope}

    \begin{scope}[yshift=-75]
    \draw[black] (-1,0)--(1,0)--(.5,{sqrt(3)/2})--(0,0)--(-.5,{sqrt(3)/2})--cycle;
    \draw[black] (0,0)--(-.5,{-sqrt(3)/2})--(.5,{-sqrt(3)/2})--cycle;
    \node[fill=gray] at (-1,0) {};
    \node[fill=white] at (0,0) {};
    \node[fill=gray] at (1,0) {};
    \node[fill=white] at (0.5,{sqrt(3)/2}) {};
    \node[fill=white] at (-0.5,{sqrt(3)/2}) {};
    \node[fill=gray] at (-0.5,{-sqrt(3)/2}) {};
    \node[fill=white] at (0.5,{-sqrt(3)/2}) {};
    \end{scope}
    
    \begin{scope}[xshift=75,yshift=-75]
    \draw[black] (-1,0)--(1,0)--(.5,{sqrt(3)/2})--(0,0)--(-.5,{sqrt(3)/2})--cycle;
    \draw[black] (0,0)--(-.5,{-sqrt(3)/2})--(.5,{-sqrt(3)/2})--cycle;
    \node[fill=gray] at (-1,0) {};
    \node[fill=white] at (0,0) {};
    \node[fill=gray] at (1,0) {};
    \node[fill=white] at (0.5,{sqrt(3)/2}) {};
    \node[fill=white] at (-0.5,{sqrt(3)/2}) {};
    \node[fill=white] at (-0.5,{-sqrt(3)/2}) {};
    \node[fill=gray] at (0.5,{-sqrt(3)/2}) {};
    \end{scope}

    \begin{scope}[xshift=150,yshift=-75]
    \draw[black] (-1,0)--(1,0)--(.5,{sqrt(3)/2})--(0,0)--(-.5,{sqrt(3)/2})--cycle;
    \draw[black] (0,0)--(-.5,{-sqrt(3)/2})--(.5,{-sqrt(3)/2})--cycle;
    \node[fill=gray] at (-1,0) {};
    \node[fill=white] at (0,0) {};
    \node[fill=white] at (1,0) {};
    \node[fill=gray] at (0.5,{sqrt(3)/2}) {};
    \node[fill=white] at (-0.5,{sqrt(3)/2}) {};
    \node[fill=gray] at (-0.5,{-sqrt(3)/2}) {};
    \node[fill=white] at (0.5,{-sqrt(3)/2}) {};
    \end{scope}

    \begin{scope}[xshift=225,yshift=-75]
    \draw[black] (-1,0)--(1,0)--(.5,{sqrt(3)/2})--(0,0)--(-.5,{sqrt(3)/2})--cycle;
    \draw[black] (0,0)--(-.5,{-sqrt(3)/2})--(.5,{-sqrt(3)/2})--cycle;
    \node[fill=gray] at (-1,0) {};
    \node[fill=white] at (0,0) {};
    \node[fill=white] at (1,0) {};
    \node[fill=gray] at (0.5,{sqrt(3)/2}) {};
    \node[fill=white] at (-0.5,{sqrt(3)/2}) {};
    \node[fill=white] at (-0.5,{-sqrt(3)/2}) {};
    \node[fill=gray] at (0.5,{-sqrt(3)/2}) {};
    \end{scope}

    \begin{scope}[xshift=25,yshift=-150]
    \draw[black] (-1,0)--(1,0)--(.5,{sqrt(3)/2})--(0,0)--(-.5,{sqrt(3)/2})--cycle;
    \draw[black] (0,0)--(-.5,{-sqrt(3)/2})--(.5,{-sqrt(3)/2})--cycle;
    \node[fill=white] at (-1,0) {};
    \node[fill=gray] at (0,0) {};
    \node[fill=gray] at (1,0) {};
    \node[fill=gray] at (0.5,{sqrt(3)/2}) {};
    \node[fill=white] at (-0.5,{sqrt(3)/2}) {};
    \node[fill=gray] at (-0.5,{-sqrt(3)/2}) {};
    \node[fill=gray] at (0.5,{-sqrt(3)/2}) {};
    \end{scope}
    
    \begin{scope}[xshift=113,yshift=-150]
    \draw[black] (-1,0)--(1,0)--(.5,{sqrt(3)/2})--(0,0)--(-.5,{sqrt(3)/2})--cycle;
    \draw[black] (0,0)--(-.5,{-sqrt(3)/2})--(.5,{-sqrt(3)/2})--cycle;
    \node[fill=gray] at (-1,0) {};
    \node[fill=gray] at (0,0) {};
    \node[fill=white] at (1,0) {};
    \node[fill=white] at (0.5,{sqrt(3)/2}) {};
    \node[fill=gray] at (-0.5,{sqrt(3)/2}) {};
    \node[fill=gray] at (-0.5,{-sqrt(3)/2}) {};
    \node[fill=gray] at (0.5,{-sqrt(3)/2}) {};
    \end{scope}

    \begin{scope}[xshift=201,yshift=-150]
    \draw[black] (-1,0)--(1,0)--(.5,{sqrt(3)/2})--(0,0)--(-.5,{sqrt(3)/2})--cycle;
    \draw[black] (0,0)--(-.5,{-sqrt(3)/2})--(.5,{-sqrt(3)/2})--cycle;
    \node[fill=gray] at (-1,0) {};
    \node[fill=gray] at (0,0) {};
    \node[fill=gray] at (1,0) {};
    \node[fill=gray] at (0.5,{sqrt(3)/2}) {};
    \node[fill=gray] at (-0.5,{sqrt(3)/2}) {};
    \node[fill=white] at (-0.5,{-sqrt(3)/2}) {};
    \node[fill=white] at (0.5,{-sqrt(3)/2}) {};
    \end{scope}
    \end{tikzpicture}
    \caption{All $2^3+3\binom{2}{2} = 11$ minimal forts (white) of $Wd(3,3)$. The first two rows depict the 8 minimal forts containing the cut vertex, while the last row has the 3 minimal forts that do not contain the cut vertex.}
    \label{fig:forts_Wd(3,3)}
\end{figure}
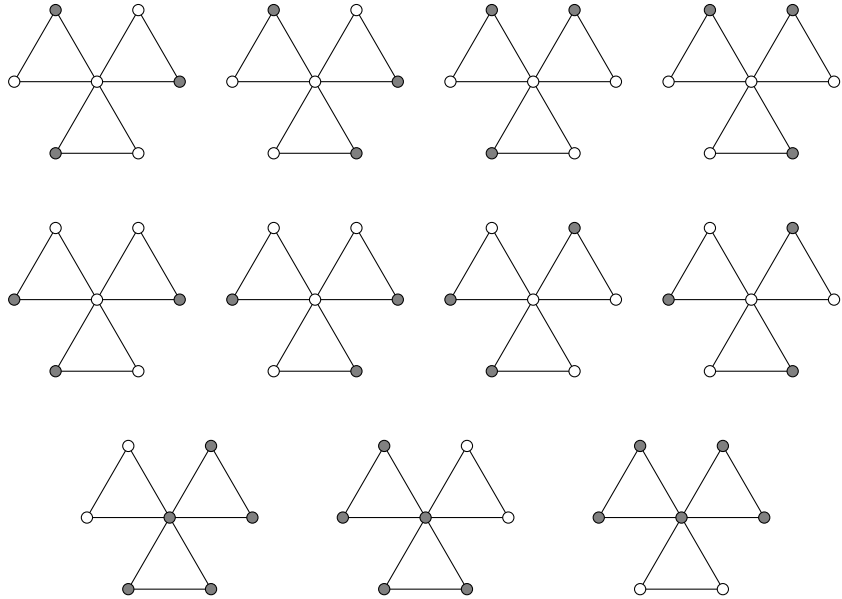

Note that the windmill graph $Wd(r,k)$ has order $n=k(r-1)+1$.
Therefore, Corollary~\ref{cor:windmill} implies that the number of minimal forts of $Wd(r,k)$ satisfies
\[
\abs{\mathcal{F}_{Wd(r,k)}} = (r-1)^{\frac{n-1}{r-1}} + \frac{n-1}{r-1}\binom{r-1}{2},
\]
which is exponential in the order of the graph. 
Moreover, the base of the exponent, $(r-1)^{\frac{1}{r-1}}$, is maximized over discrete values of $r$ when $r=4$.
In this case, the base of the exponent is $3^{1/3}=1.4422495703\ldots$, which is larger than the plastic ratio. 
Hence, the number of minimal forts of the windmill graph $Wd(4,k)$ grows faster than the cycle graph and path graph of the same order, see Corollaries~\ref{cor:cycle_graph_explicit} and~\ref{cor:path_graph_explicit}, respectively.
We summarize this result in the following corollary.
\begin{corollary}\label{cor:windmill4k}
The number of minimal forts for $Wd(4,k)$ satisfies
\[
\abs{\mathcal{F}_{Wd(4,k)}} = 3^{\frac{n-1}{3}} + (n-1),
\]
where $n=3k+1$ is the order of $Wd(4,k)$. 
Moreover, the limit of successive ratios satisfies
\[
\lim_{k\rightarrow\infty}\frac{\abs{\mathcal{F}_{Wd(4,k+1)}}}{\abs{\mathcal{F}_{Wd(4,k)}}} = 3^{1/3}. 
\]
\end{corollary}

While $3^{1/3}$ is the largest asymptotic growth rate we have observed, we suspect that there exists families of graphs with a larger asymptotic growth rate, see Conjecture~\ref{con:gen_gr}.
\section{Conclusion}\label{sec:conclusion}
In this article, we investigate the number of minimal forts of a graph.
In Corollary~\ref{cor:sperner_fort}, we show that the number of minimal forts of a graph of order $n\geq 6$ is strictly less than $\binom{n}{\floor{n/2}}$, which is a famous combinatorial bound due to Emanuel Sperner (1928) on the size of a collection of subsets where no subset contains another.
Then, we derive explicit formula for the number of minimal forts of several families of graphs.
In particular, see Corollary~\ref{cor:cycle_graph} for the cycle graph, Corollary~\ref{cor:path_graph} for the path graph, and Theorem~\ref{thm:spider_graph} for the spider graph.
We also develop methods for constructing minimal forts for graph products using minimal forts of the original graphs. 
In the process, we provide explicit formula and lower bounds on the number of minimal forts for additional families of graphs.
In particular, see Corollary~\ref{cor:wheel} for the wheel graph, Corollary~\ref{cor:sunlet} for the sunlet graph, and Corollary~\ref{cor:windmill} for the windmill graph. 
A summary of explicit formula and lower bounds for the number of minimial forts of various families of graphs is shown in Table~\ref{tab:summary}. 

\begin{table}[ht]
\resizebox{1.0\textwidth}{!}{%
\begin{tabular}{l|c|c|c}
Result & $G$ & Order & $\abs{\mathcal{F}_{G}}$ \\
\hline
Proposition~\ref{prop:empty_graph} & $E_{n}$ & $n$ & $n$ \\
Proposition~\ref{prop:complete_graph} & $K_{n}$ & $n$ & $\binom{n}{2}$ \\
Proposition~\ref{prop:complete_bipartite} & $K_{p,q}$ & $p+q$ & $\binom{p}{2} + \binom{q}{2}$ \\
Corollary~\ref{cor:cycle_graph_explicit} & $C_{n}$ & $n$ & $\nint{\psi^{n}},~n\geq 10$ \\
Corollary~\ref{cor:path_graph_explicit} & $P_{n}$ & $n$ & $\nint{\frac{\psi^{n+4}}{2\psi+3}},~n\geq 1$ \\
Theorem~\ref{thm:spider_graph} & $S_{l_{1},\ldots,l_{k}}$ & $\sum_{i=1}^{k}l_{i}+1$ & $\prod_{i=1}^{k}\abs{\mathcal{F}_{P_{l_{i}-1}}} + \sum_{i=1}^{k}\abs{\mathcal{F}_{P_{l_{i}-2}}}\prod_{j\neq i}\abs{\mathcal{F}_{P_{l_{j}-1}}}+\sum_{1\leq i<j\leq k}\abs{\mathcal{F}_{P_{l_{i}}}}\abs{\mathcal{F}_{P_{l_{j}}}}$ \\
Corollary~\ref{cor:wheel} & $C_{n-1}\vee K_{1}$ & $n$ & $\geq\nint{\psi^{n-1}},~n\geq 11$ \\
Corollary~\ref{cor:sunlet} & $C_{n}\circ K_{1}$ & $2n$ & $\geq\nint{\psi^{n}},~n\geq 10$ \\
Corollary~\ref{cor:windmill} & $Wd(r,k)$ & $k(r-1)+1$ & $k\binom{r-1}{2}+(r-1)^{k}$ \\
Corollary~\ref{cor:windmill4k} & $Wd(4,k)$ & $n=3k+1$ & $\geq 3^{\frac{n-1}{3}}$
\end{tabular}%
}
\caption{Summary of the number of minimal forts for several families of graphs.}
\label{tab:summary}
\end{table}

In Theorem~\ref{thm:spider_graph_bound}, we prove that the number of minimal forts for the spider graph $S_{l_{1},\ldots,l_{k}}$ is bounded above by $\binom{k}{2}\abs{\mathcal{F}_{P_{n}}}$, where $n=\sum_{i=1}^{k}l_{i}+1$ is the order of the spider graph.
Therefore, the asymptotic growth rate of the spider graph is bounded above by the asymptotic growth rate of the path graph, which is the plastic ratio.
We conjecture that the path graph has an extremal number of minimal forts over all trees. 
\begin{conjecture}\label{con:tree_gr}
Let $T_{n}$ denote a tree with the largest number of minimal forts over all trees of order $n\geq 1$. 
Then, the number of minimal forts satisfies
\[
\abs{\mathcal{F}_{T_{n}}} \leq \binom{n}{2}\abs{\mathcal{F}_{P_{n}}}.
\]
In particular, the limit of successive ratios satisfies
\[
\lim_{n\rightarrow\infty}\frac{\abs{\mathcal{F}_{T_{n+1}}}}{\abs{\mathcal{F}_{T_{n}}}} = \psi. 
\]
\end{conjecture}

In Corollary~\ref{cor:windmill4k}, we prove that the number of minimal forts of $Wd(4,k)$ has an asymptotic growth rate of $3^{1/3}$.
It is worth noting that $3^{1/3}$ happens to be the maximum asymptotic growth rate of maximal cliques and independent sets~\cite{Griggs1988,Moon1965,Wilf1986}.
We conjecture that there exist families of graphs where the number of minimal forts has an asymptotic growth rate strictly greater than $3^{1/3}$.
\begin{conjecture}\label{con:gen_gr}
There exists a family of graphs $G_{n}$ such that
\[
3^{1/3} < \lim_{n\rightarrow\infty}\frac{\abs{\mathcal{F}_{G_{n+1}}}}{\abs{\mathcal{F}_{G_{n}}}} < 2,
\]
where $2$ is the asymptotic growth rate of Sperner's bound. 
\end{conjecture}
\subsection*{Acknowledgements}
Thanks to Professor Boris Brimkov for pointing out Theorem 1 and Example 2 of~\cite{Brimkov2021}. 

\end{document}